\documentclass[a4paper,11pt, reqno]{amsart} 
\usepackage{amssymb,amsthm,amsmath}

\usepackage{amsfonts}
\usepackage{amscd}
\usepackage{amssymb}
\usepackage{enumerate}
\usepackage{graphicx}
\allowdisplaybreaks
\usepackage{mathabx}
\usepackage{color}
\usepackage{amsbsy}
\usepackage{graphicx}
\usepackage{amsthm}
\usepackage{amsmath}
\usepackage{amsxtra}
\usepackage{mathrsfs}
\usepackage{bbm}
\usepackage{dsfont}

\usepackage{ifthen}

\usepackage[colorlinks,citecolor=red,pagebackref,hypertexnames=false]{hyperref} 

\newtheorem{theorem}{Theorem}[section]
\newtheorem{defi}[theorem]{Definition}
\newtheorem{remark}[theorem]{Remark}

\newtheorem{prop}[theorem]{Proposition}

\numberwithin{equation}{section}
\definecolor{red}{rgb}{1.0, 0.0, 0.0}
\newcommand{\Bea}{\begin{eqnarray*}}
	\newcommand{\Eea}{\end{eqnarray*}}
\newcommand{\Be} {\begin{equation*}}
	\newcommand{\Ee} {\end{equation*}}
\newcommand{\be} {\begin{equation}}
	\newcommand{\ee} {\end{equation}}
\newcommand{\bea} {\begin{eqnarray}}
	\newcommand{\eea} {\end{eqnarray}}

\title[ Non-harmonic analysis of the wave equation for Schr\"{o}dinger operators with complex potential]{ Non-harmonic analysis of the wave equation for Schr\"{o}dinger operators with complex potential}
\author[Aparajita Dasgupta]{Aparajita Dasgupta}
\address{
	Aparajita Dasgupta:
	\endgraf
	Department of Mathematics
	\endgraf
	Indian Institute of Technology, Delhi, Hauz Khas
	\endgraf
	New Delhi-110016 
	\endgraf
	India
	\endgraf
	{\it E-mail address:} {\rm adasgupta@maths.iitd.ac.in}
}
\author[Lalit Mohan]{Lalit Mohan}
\address{
	Lalit Mohan:
	\endgraf
	Department of Mathematics
	\endgraf
	Indian Institute of Technology, Delhi, Hauz Khas
	\endgraf
	New Delhi-110016 
	\endgraf
	India
	\endgraf
	{\it E-mail address:} {\rm mohanlalit871@gmail.com}
}
\author[Shyam Swarup Mondal]{Shyam Swarup Mondal}
\address{
	Shyam Swarup Mondal:
	\endgraf
	Stat-Math unit
	\endgraf
	Indian Statistical Institute Kolkata 
	\endgraf
	BT Road,  Baranagar, Kolkata 700108
	\endgraf
	India
	\endgraf
	{\it E-mail address:} {\rm mondalshyam055@gmail.com}
}

\date{\today}

\vspace{1cm}

\allowdisplaybreaks
\begin{document}
	
	\begin{abstract}
	This article investigates the wave equation for the Schr\"{o}dinger operator on $\mathbb{R}^{n}$, denoted as $\mathcal{H}_0:=-\Delta+V$, where $\Delta$ is the standard Laplacian and $V$ is a complex valued multiplication operator.  We prove that the operator  $\mathcal{H}_0$,  with  $\operatorname{Re}(V)\geq 0$ and $\operatorname{Re}(V)(x)\to\infty$ as $|x|\to\infty$, has a purely discrete spectrum under certain conditions.  In the spirit of Colombini, De Giorgi, and Spagnolo, we also prove that the Cauchy problem with regular coefficients is well-posed in the associated Sobolev spaces, and when the propagation speed is H\"{o}lder continuous (or more regular), it is well-posed in Gevrey spaces. Furthermore, we prove that it is very weakly well-posed when the coefficients possess a distributional singularity.
	\end{abstract}
	\maketitle
	\tableofcontents
	
		\section{Introduction}\label{sec1 ch5}

In this paper, we investigate the well-posedness of the Cauchy problem for time-dependent wave equations associated to a Schr\"{o}dinger operator with complex potential.  Let $\mathcal{H}_{0}$ be the  Schr\"{o}dinger operator on $\mathbb{R}^n$ with complex potential $V$ defined by
\begin{equation}\label{schrodinger operator defi ch5}
	\mathcal{H}_0 u(x)= (-\Delta + V) u(x), \quad x \in \mathbb{R}^n,
\end{equation}
where $\Delta$ is the usual Laplacian on $\mathbb{R}^n$.

For a non-negative function $a=a(t)\geq 0$ and for the source term $f=f(t,x)$, we are interested in the the Cauchy problem associated for the operator $\mathcal{H}_{0}$ with the propagation speed $a$, namely
	\begin{equation}\label{initial cauchy problem ch5}
		\left\{\begin{array}{l}
			\partial_t^2 v(t, x)+a(t)\mathcal{H}_0v(t, x)+q(t)v(t, x)=f(t, x),\quad (t, x) \in[0, T] \times \mathbb{R}^n, \\
			v(0, x)=v_0(x),\quad x \in  \mathbb{R}^n, \\
			\partial_t v(0, x) =v_1(x),\quad x \in \mathbb{R}^n,
		\end{array}\right.
\end{equation}
where $q$ is a bounded real-valued function.  In order to provide a detailed analysis of these problems, we will examine various cases based primarily on the properties of the propagation speed  $a,$ and to a lesser extent, those of the source term
$f$.   This approach is crucial because the optimal outcomes differ depending on the characteristics of $a$.  Particularly, we will consider the following cases:
\begin{itemize} 
       \item {Case} 1: The coefficient $a$ is regular enough: $a \in L_{1}^{\infty} ([0, T])$ be such that $\underset{t \in [o,T]}{\inf} a(t)=a_{0}>0$,
       
      \item {Case} 2: $a \in \mathcal{C}^\alpha([0, T])$, with $0<\alpha<1 $ such that $\underset{t \in [o,T]}{\inf} a(t)=a_{0}>0$,
      
    \item  {Case} 3: $a \in \mathcal{C}^l([0, T])$, with $l \geq 2$ and $ a(t) \geq 0$,
    
    \item {Case} 4:  $a \in \mathcal{C}^\alpha([0, T])$, with $0<\alpha<2$ and $ a(t) \geq 0$.
\end{itemize}

    

Here, it's important to note that when  $V$ is a complex potential, the operator $\mathcal{H}_0$ is non-self-adjoint; see \cite{RR09}. However,  in Theorem \ref{discrte spectrum result ch5}, we proved that the operator $\mathcal{H}_0$ has discrete spectrum with its eigenfunctions yielding a bi-orthogonal basis of $L^2.$  General bi-orthogonal systems, explored by Bari \cite{bari}, provide a suitable framework for our constructions; see also Gelfand \cite{gelfand}. Similar systems, though slightly more general yet fundamentally the same, are referred to as `Hilbert systems' by Bari \cite{bari} and `Quasi-Orthogonal Systems' by Kac, Salem, and Zygmund \cite{kac}. Here the description of appearing function spaces is performed in the spirit of   employing  general development of non-harmonic type analysis carried out by the authors in \cite{Ruz16}.
 Notably, as illustrated in \cite{AD&VK&LM&SSM,Ruz16,ruzhansky2017nonharmonic,MR&JPVR}, there are intriguing applications of non-harmonic theory in partial differential equations, particularly concerning the wave equation \cite{ruzhansky2017very, ruzhansky2017wave}. Building on Seeley's non-harmonic analysis on compact manifolds \cite{seeley1,seeley2}, the non-harmonic analysis of boundary value problems is  akin to global pseudo-differential analysis on closed manifolds, as described in \cite{delgado2,delgado1,DMT17}. This analysis effectively addresses numerous issues, such as characterizing classes of functions and distributions of the Komatsu type \cite{apara}. Furthermore, it provides a mathematical framework for studying linear operators within the context of non-self-adjoint quantum mechanics \cite{inoue2016non,mostafazadeh2010pseudo}. The approach to addressing global well-posedness in the relevant function spaces in this paper extends the method developed in \cite {garetto micheal} within the context of compact Lie groups.

The family of linear operators applies to various potentials $V$ and serves to distinguish various quantum systems. We recall a few fundamental quantum systems along with their respective spectra.
 \begin{itemize}
     \item The Schr\"{o}dinger operator, characterized by free motion (i.e., the absence of any external force on the electrons), is represented by the Laplacian function that possesses a continuous spectrum inside the positive real axis. 
     \item The Coulomb potential, that defines the Schr\"{o}dinger operator of a hydrogen atom with an infinitely heavy nucleus positioned at the origin,  is given by 
	$$
	\mathcal{H}_0= -\Delta + \frac{1}{|x|^2}.
	$$
	It has the essential spectrum contained in the positive half real-axis, while the negative real-axis contains isolated eigenvalues with finite multiplicity. 
 \item A totally discrete spectrum is exhibited by the Schr\"{o}dinger operator, which is a quantum harmonic oscillator with potential $V (x) = |x|^2$ and an anharmonic oscillator with potential $V (x, y) = x^2 y^2$.
 \end{itemize}    For more information about Schr\"{o}dinger operators, we   refer to \cite{Dziubanski,fernandez,gustafson,hislop}. In addition, if the potential $V: \mathbb{R}^n \rightarrow \mathbb{R}$ be a non-negative continuous function such that $|V(x)| \rightarrow \infty$ as $|x| \rightarrow \infty$, then the essential spectrum of $\mathcal{H}_0$, 
$\sigma_{\mathrm{ess}}\left(\mathcal{H}_0\right)=\emptyset$, i.e., the operator $\mathcal{H}_0$  has completely discrete spectrum, see \cite{gustafson}.   In the direction of discrete lattice,  there is limited literature available concerning the study of the spectrum of discrete Schr\"{o}dinger operator on $\hbar\mathbb{Z}^{n}$.    Regarding the discrete spectrum for the discrete Schr\"{o}dinger operator with potential, one can see \cite{abhilash}. We also refer to    \cite{Rab10, Rab13, RR04, RR06, RR09} regarding the spectrum for the Schr\"{o}dinger operator with or without potential on the discrete lattice  $\mathbb{Z}^{n}$. Here, the authors carried a study on essential spectrum of Schr\"{o}dinger operators that have potentials which oscillate slowly as they approach infinity and approximate the eigenfunctions associated with the corresponding Schr\"{o}dinger operators. They also examined operators with periodic and semi-periodic potentials. Additionally, they investigated Schr\"{o}dinger operators that serve as discrete quantum analogs of acoustic propagators for waveguides. 

In the papers referred to above, Schrödinger operators with real potentials have been examined. In contrast, our paper focuses on Schrödinger operators with complex potentials, which are non-self-adjoint operators. To study the spectrum of such operators, one of the essential ingredients is Weyl's theorem, which asserts that if $A$
 is a self-adjoint operator and $B$ is
 symmetric and compact,  then $ess(A+B)=ess(A)$.  This was first proved  (in a special case) by  H. Weyl in \cite{weyl} and is now commonly known as Weyl's theorem.    Later, M. Schechter   \cite{Schechter} established an alternative version of Weyl's theorem for non-self-adjoint operators  under the assumption that $B$ is relatively compact with respect to 
$A$. In particular, in this paper, we have demonstrated using Schechter's result that the spectrum of Schr\"{o}dinger operator $\mathcal{H}_{0}$  with complex potentials is discrete. More specifically, regarding the spectrum of  $\mathcal{H}_{0}$, we have the following result.
\begin{theorem} 
	Let $\operatorname{Re}(V)\geq 0$ and $\operatorname{Re}(V)(x) \rightarrow \infty$ as $|x| \rightarrow \infty$. Also, assume that $i\operatorname{Im}(V)$ is relatively compact with respect to the operator $\mathcal{H}_{00}$, where $\mathcal{H}_{00}=-\Delta + \operatorname{Re}(V)$. Then the operator $\mathcal{H}_{0}$ has a purely discrete spectrum.
\end{theorem}
A detailed proof of the above theorem is provided in Section \ref{spectrum and fourier analysis ch5}. We further observe that the condition requiring the spectru of $\mathcal{H}_0$ to be non-negative can be partially relaxed. Specifically, let $\mathcal{H}_0$ be the densely defined Schr\"{o}dinger operator on $L^2(\mathbb{R}^n)$ with a discrete spectrum $\{\lambda_{\xi}\in \mathbb{C} : \xi\in \mathcal{I}\}$, with corresponding eigenfunctions forming basis of $L^2(\mathbb{R}^n)$, where $I$ is an ordered countable set. Define $\mathcal{H}:=|\mathcal{H}_0|$ as the operator with eigenvalue $|\lambda_\xi|$ for each eigenfunction $e_{\xi}$. If $\lambda_\xi=0$ for some $\xi$, for instance, we consider $\mathcal{H} := |\mathcal{H}_1|$,  the operator with eigenvalue $|\lambda_\xi|+c$ for each eigenfunction $e_\xi$, where $c>0$ is a positive constant.

It should be noted that $\mathcal{H}$ is not the absolute value of $\mathcal{H}_0$ in the operator sense, as $\mathcal{H}_0$ and its adjoint $\mathcal{H}_0^{\ast}$, may have different domains and are generally not composable. However, this definition is well-established through the symbolic calculus developed in \cite{Ruz16} and extended in \cite{ruzhansky2017nonharmonic} to encompass the full pseudo-differential calculus, even without the condition that eigenfunctions have no zeros. Therefore, in this paper, we study the following Cauchy problem, which is a modification of Cauchy problem \eqref{initial cauchy problem ch5} with $\mathcal{H}=|\mathcal{H}_{0}|$:
\begin{equation}\label{cauchy problem ch5}
		\left\{\begin{array}{l}
			\partial_t^2 v(t, x)+a(t)\mathcal{H}v(t, x)+q(t)v(t, x)=f(t, x),\quad (t, x) \in[0, T] \times \mathbb{R}^n, \\
			v(0, x)=v_0(x),\quad x \in  \mathbb{R}^n, \\
			\partial_t v(0, x) =v_1(x),\quad x \in \mathbb{R}^n.
		\end{array}\right.
\end{equation}

	Returning to the primary focus of this article, numerous authors have conducted significant studies on the Cauchy problem of the form \eqref{cauchy problem ch5}. Particularly, to deal with regular coefficients and source terms, one can refer to works \cite{colombini2013,colombini1979,colombini2002,colombini2003}. On the other hand, for distributional irregularities, one may consider $q$ to be the $\delta$-distribution when electric potential induces shocks or to be a Heaviside function when the speed of propagation is discontinuous. In particular, here, we are interested in the appearance of distributional irregularities like $\delta$, $\delta^2$, $\delta\delta^{\prime}$ etc. The impossibility of multiplying distributions poses significant mathematical challenges for the conventional distributional interpretation of the equation and the Cauchy problem, as highlighted by Schwartz \cite{schwartz}. However, one can solve this issue by establishing the well-posedness of the problem using the concept of very weak solutions, as introduced in \cite{garetto}, within the framework of space-invariant hyperbolic problems. Subsequently, the implementation of several physical models was carried out in \cite{garetto2021,garetto micheal, ruzhansky2017wave}. Also, the wave equations related to operators with a purely discrete spectrum and coefficients with distributional irregularities have been thoroughly investigated in \cite{ancona,abhilash,abhilash2,ruzhansky2017very,ruzhansky2017wave,ruzhansky2019wave}. 
 
 Apart from the introduction, the organization of the paper is as follows.
 \begin{itemize}
     \item In Section \ref{main result sec2 ch5}, we discuss results concerning  the Case 1 - Case 4.

     \item In Section \ref{spectrum and fourier analysis ch5},  we will review the basic concepts of harmonic analysis within the context of non-harmonic analysis. Particularly, we recall the concepts and techniques provided by M. Ruzhansky and N. Tokmagambetov  in \cite{Ruz16} 

     \item   In Section \ref{proof of results ch5}, we prove main results for each Case 1  to Case 4.

     \item In Section \ref{very weak solutions section ch5}, we discuss the case of distributional irregularities, i.e., we allow the coefficients to be irregular and investigate the existence of very weak solutions to the Cauchy problem 
 (\ref{cauchy problem ch5}).
 \end{itemize}

	\section{Main results: Comprehensive review}\label{main result sec2 ch5}
Throughout the paper, we denote $L_1^{\infty}([0, T])$, the space of all  such   $a \in L^{\infty}([0, T])$ differentiable function which satisfies the condition $\partial_t a \in L^{\infty}([0, T])$.  Also,  the Sobolev space $H_{\mathcal{H}}^s, s \in \mathbb{R},$ associated to the operator $\mathcal{H}$,   is  defined as
	$$
	H_{\mathcal{H}}^s:=\left\{f \in \mathcal{D}_{\mathcal{H}}^{\prime}\left(\mathbb{R}^n\right): \mathcal{H}^{s / 2} f \in L^2\left(\mathbb{R}^n\right)\right\}
	$$
equipped	with the norm $\|f\|_{H_{\mathcal{H}}^s}:=\left\|\mathcal{H}^{s / 2} f\right\|_{L^2}$ and the space of distributions $\mathcal{D}_{\mathcal{H}}^{\prime}\left(\mathbb{R}^n\right)$ is defined in Section \ref{spectrum and fourier analysis ch5}.

	First, we investigate   Cauchy problem \eqref{cauchy problem ch5} involving regular coefficients $a \in L_{1}^{\infty}([0, T])$ and $q \in L^{\infty}([0, T])$ and obtain   the following well-posedness result  in the Sobolev spaces $H_{\mathcal{H}}^s $ to the Cauchy problem \eqref{cauchy problem ch5}.
\begin{theorem}\textbf{(Case 1)}\label{classical sol case1 ch5}
 Let $T>0$ and $s \in \mathbb{R}$. Assume that $a \in L_{1}^{\infty} ([0, T])$ be such that $\underset{t \in [o,T]}{\inf} a(t)=a_{0}>0$ and $q \in L^{\infty}([0, T])$. Let $f \in L^{2}\left([0, T] ; \mathrm{H}_{\mathcal{H}}^{s}\right)$. If the initial Cauchy data $\left(v_{0}, v_{1}\right) \in \mathrm{H}_{\mathcal{H}}^{1+s} \times \mathrm{H}_{\mathcal{H}}^{s}$, then the Cauchy problem \eqref{cauchy problem ch5} has a unique solution $v \in C\left([0, T] ; \mathrm{H}_{\mathcal{H}}^{1+s}\right) \bigcap C^{1}\left([0, T] ; \mathrm{H}_{\mathcal{H}}^{s}\right)$ satisfying the estimate
 \begin{equation}\label{classical sol estimate ch5}
 \|v(t, \cdot)\|_{\mathrm{H}_{\mathcal{H}}^{1+s}}^{2}+\left\|v_{t}(t, \cdot)\right\|_{\mathrm{H}_{\mathcal{H}}^{s}}^{2} \leq C\left(\left\|v_{0}\right\|_{\mathrm{H}_{\mathcal{H}}^{1+s}}^{2}+\left\|v_{1}\right\|_{\mathrm{H}_{\mathcal{H}}^{s}}^{2}+\|f\|_{L^{2}\left([0, T] ; \mathrm{H}_{\mathcal{H}}^{s}\right)}^{2}\right),
 \end{equation}
for all $t \in[0, T]$, where the constant $C>0$ independent of $t \in[0, T]$.
	\end{theorem}
 We know that, in the case of $\mathbb{R}^n$, when we deal with PDE's, the well-posedness in the spaces of smooth functions or distributions fails in cases when $a$ is H\"{o}lder continuous or $a$ is positive but need not be away from zero. For instance, it is feasible to obtain smooth Cauchy data such that Cauchy problem \eqref{cauchy problem ch5} does not possess solutions in distribution spaces, or if a solution does exist, it may not be unique. For construction of such examples, one can refer to \cite{colombini1987,colombini1982}. Thus, Gevrey spaces naturally arise in such setting. For real $s\geq 1$, we define Gevrey spaces $\gamma_{\mathcal{H}}^{s}$  in our setting formally in Section \ref{spectrum and fourier analysis ch5}. Now, we will examine a situation where the propagation speed, $a(t)\geq a_{0}>0$, is H\"{o}lder continuous of order $\alpha,$ with $0<\alpha<1.$ Note that, for $\alpha>0$, we say that $a$ is H\"{o}lder continuous of order $\alpha,$ i.e., $a \in \mathcal{C}^\alpha([0, T])$ if for some positive constant $M_a$, we have
$$
|a(t)-a(s)| \leq M_a|t-s|^\alpha,
$$
for all $t, s \in[0, T]$. We denote $\|a\|_{\mathcal{C}^\alpha([0, T])}$ as the smallest $M_a$ in the above inequality.
 \begin{theorem}\textbf{(Case 2)}\label{classical sol case2 ch5}
 Let $T>0$ and $s \in \mathbb{R}$. Assume that  $a \in \mathcal{C}^\alpha([0, T])$, with $0<\alpha<1,$ be such that $\underset{t \in [o,T]}{\inf} a(t)=a_{0}>0$ and $q \in L^{\infty}([0, T])$. Let $f \in C\left([0, T] ; \gamma_{\mathcal{H}}^{s}\right)$. If the initial Cauchy data  $v_{0}, v_{1} \in \gamma_{\mathcal{H}}^{s}$, then the Cauchy problem \eqref{cauchy problem ch5} has a unique solution $v \in C^2\left([0, T] ; \gamma_{\mathcal{H}}^{s}\right)$ provided that
 $$1\leq s < 1+ \frac{\alpha}{1-\alpha}.$$
	\end{theorem}
 Now, we will examine a situation where the propagation speed $a(t)$, can potentially reach 0 but remains consistent, i.e., $a \in \mathcal{C}^l([0, T])$, with $l\geq2.$ 
 \begin{theorem}\textbf{(Case 3)}\label{classical sol case3 ch5}
 Let $T>0$ and $s \in \mathbb{R}$. Assume that $a \in \mathcal{C}^l([0, T])$ with $l\geq2,$  such that $a(t)\geq0$ and $q \in L^{\infty}([0, T])$. Let $f \in C\left([0, T] ; \gamma_{\mathcal{H}}^{s}\right)$. If the initial Cauchy data $v_{0}, v_{1} \in \gamma_{\mathcal{H}}^{s}$, then the Cauchy problem \eqref{cauchy problem ch5} has a unique solution $v \in C^2\left([0, T] ; \gamma_{\mathcal{H}}^{s}\right)$ provided that
 $$1\leq s < 1+ \frac{l}{2}.$$
	\end{theorem}
 Finally, we examine the case that is complementary to the one described in Theorem \ref{classical sol case3 ch5}. In this instance, the propagation speed $a(t)$ has the possibility of reaching zero and is not as smooth or regular, i.e., $a \in \mathcal{C}^\alpha([0, T])$ with $0<\alpha<2.$
 \begin{theorem}\textbf{(Case 4)}\label{classical sol case4 ch5}
 Let $T>0$ and $s \in \mathbb{R}$. Assume that $a \in \mathcal{C}^\alpha([0, T])$ with $0<\alpha<2,$   such that $a(t)\geq0$ and $q \in L^{\infty}([0, T])$. Let $f \in C\left([0, T] ; \gamma_{\mathcal{H}}^{s}\right)$. If the initial Cauchy data $v_{0}, v_{1} \in \gamma_{\mathcal{H}}^{s}$, then the Cauchy problem \eqref{cauchy problem ch5} has a unique solution $v \in C^2\left([0, T] ; \gamma_{\mathcal{H}}^{s}\right)$ provided that
 $$1\leq s < 1+ \frac{\alpha}{2}.$$
	\end{theorem}

\section{Preliminaries: Non harmonic analysis}\label{spectrum and fourier analysis ch5}
In this section, our first aim is to find condition on the complex potential $V$ such that the operator $\mathcal{H}_{0}$ has purely discrete spectrum. We begin with the following definition:
\begin{defi}\label{relative compact defi ch5}
	 Let $X$ be a Hilbert space. Let $T: X \rightarrow X$ be a linear operator with dense domain $\mathcal{D}\left(T\right)$ in $X$ and the resolvent of $T, \rho\left(T\right)$ be non-empty. We say that a linear operator $A: X \rightarrow X$ is relatively compact with respect to $T$ or $T$-compact if $\mathcal{D}\left(T\right) \subset$ $\mathcal{D}(A)$ and $A R\left(z, T\right): X \rightarrow X$ is a compact operator for some $z \in \rho\left(T\right)$, where $R\left(z, T\right)= (z-T)^{-1}.$
\end{defi}
Let $\mathcal{H} := |\mathcal{H}_{0}|$ the operator defined with the set of eigenvalues $\{|\lambda_{\xi}|; \lambda_{\xi}~ \text{is an eigenvalue of}~ \mathcal{H}_{0}\}.$
\begin{theorem}\label{discrte spectrum result ch5}

 Let $\operatorname{Re}(V)\geq 0$ and $\operatorname{Re}(V)(x) \rightarrow \infty$ as $|x| \rightarrow \infty$. Also, assume that $i\operatorname{Im}(V)$ is relatively compact with respect to $\mathcal{H}_{00}$, where $\mathcal{H}_{00}=-\Delta + \operatorname{Re}(V)$. Then the operator $\mathcal{H}_{0}$ has a purely discrete spectrum.
\end{theorem}
\begin{proof}
    By \cite[Theorem 7.5]{gustafson}, we can conclude that the operator $\mathcal{H}_{00}=-\Delta + \operatorname{Re}(V)$ is a self adjoint operator and it has a purely discrete spectrum, i.e., $ess(\mathcal{H}_{00})=\emptyset$. Since, $i\operatorname{Im}(V)$ is relatively compact with respect to $\mathcal{H}_{00}$,  so by using \cite[Theorem 2.1]{Schechter}, we have that $ess(\mathcal{H}_{0})=ess(\mathcal{H}_{00})$. Hence, $\mathcal{H}_{0}$ has a purely discrete spectrum.
\end{proof}

Let $\mathcal{I}$ be an ordered countable set. We consider the spectrum $\{|\lambda_\xi|: \xi \in \mathcal{I}\}$  of $\mathcal{H}$ with corresponding eigenfunctions in $L^2(\mathbb{R}^n)$ denoted by $u_\xi$, i.e., $$\mathcal{H}u_\xi=\lambda_\xi u_\xi~\text{ in $\mathbb{R}^n$, for all $\xi \in \mathcal{I}.$}~$$ The conjugate spectral problem is
	$$\mathcal{H}^*v_\xi=\overline{\lambda_\xi} v_\xi~\text{ in $\mathbb{R}^n$, for all $\xi \in \mathcal{I}$}.$$
	Also, we assume that the system $\{u_\xi\}_{\xi \in \mathcal{I}}$ is a Riesz basis for $L^2(\mathbb{R}^n)$.  Here we denotes $$\langle \xi \rangle:=(1+|\lambda_\xi|)^{\frac{1}{2}}.$$

\begin{defi}\label{test function defi ch5}
	The space of test functions for the operator $\mathcal{H}$, i.e., $C_{\mathcal{H}}^{\infty}(\mathbb{R}^n):=$ $\operatorname{Dom}\left(\mathcal{H}^{\infty}\right)$ is defined by 
	$$
	\operatorname{Dom}\left(\mathcal{H}^{\infty}\right):=\bigcap_{k=1}^{\infty} \operatorname{Dom}\left(\mathcal{H}^{k}\right),
	$$
	where $\operatorname{Dom}\left(\mathcal{H}^{k}\right)$ is the domain of $\mathcal{H}^{k}$ which is defined as
	$$
	\operatorname{Dom}\left(\mathcal{H}^{k}\right):=\left\{f \in L^{2}(\mathbb{R}^n): \mathcal{H}^{j} f \in L^{2}(\mathbb{R}^n), j=0,1,2, \ldots, k-1\right\}.
 $$ 
\end{defi} 
The Fréchet topology of $C_{\mathcal{H}}^{\infty}(\mathbb{R}^n)$ is given by the family of norms
$$
\|\varphi\|_{C_{\mathcal{H}}^{k}}:=\max _{j \leq k}\left\|\mathcal{H}^{j} \varphi\right\|_{L^{2}(\mathbb{R}^n)},~~ k \in \mathbb{N}_{0}, ~\varphi \in C_{\mathcal{H} }^{\infty}(\mathbb{R}^n).
$$
Similarly, we introduce the space $C_{\mathcal{H}^{*}}^{\infty}(\mathbb{R}^n)$ corresponding to the adjoint operator $\mathcal{H}_{\mathbb{R}^n}^{*}$ by
$$
C_{\mathcal{H}^{*}}^{\infty}(\mathbb{R}^n):=\operatorname{Dom}\left((\mathcal{H}^{*})^{\infty}\right)=\bigcap_{k=1}^{\infty} \operatorname{Dom}\left(\left(\mathcal{H}^{*}\right)^{k}\right),
$$
where $\operatorname{Dom}\left(\left(\mathcal{H}^{*}\right)^{k}\right)$ is the domain of the operator $\left(\mathcal{H}^{*}\right)^{k}$ given by
$$
\operatorname{Dom}\left(\left(\mathcal{H}^{*}\right)^{k}\right):=\left\{f \in L^{2}(\mathbb{R}^n):\left(\mathcal{H}^{*}\right)^{j} f \in L^{2}(\mathbb{R}^n), j=0,1,2, \ldots, k-1\right\}.
$$
Similarly, the Fréchet topology of $C_{\mathcal{H}^*}^{\infty}(\mathbb{R}^n)$ will given by the family of these norms
$$
\|\varphi\|_{C_{\mathcal{H}^*}^{k}}:=\max _{j \leq k}\left\|\left(\mathcal{H}^{*}\right)^{j} \varphi\right\|_{L^{2}(\mathbb{R}^n)}, ~~k \in \mathbb{N}_{0}, ~\varphi \in C_{\mathcal{H}^* }^{\infty}(\mathbb{R}^n).
$$
\begin{remark}\label{desne remark ch5}
	Since $u_\xi \in C_{\mathcal{H}}^{\infty}(\mathbb{R}^n)$ and $v_\xi \in  C_{\mathcal{H}^{*}}^{\infty}(\mathbb{R}^n),$ for each $\xi \in \mathcal{I}$, hence it can be observed that the spaces $C_{\mathcal{H}}^{\infty}(\mathbb{R}^n)$ and $C_{\mathcal{H}^*}^{\infty}(\mathbb{R}^n)$ are dense in $L^2(\mathbb{R}^n)$.	
\end{remark}
The space of $\mathcal{H}$-distributions (or $\mathcal{H}^*$-distributions) $\mathcal{D}_\mathcal{H}'(\mathbb{R}^n)$ (or $\mathcal{D}_{\mathcal{H}^*}'(\mathbb{R}^n))$ is the space of linear continuous functionals on $C_{\mathcal{H}^{*}}^{\infty}(\mathbb{R}^n)$ (or $C_{\mathcal{H}}^{\infty}(\mathbb{R}^n)$), i.e.,
$\mathcal{D}_\mathcal{H}'(\mathbb{R}^n) :=\mathcal{L}(C_{\mathcal{H}^{*}}^{\infty}(\mathbb{R}^n), \mathbb{C} )$ (or $\mathcal{D}_{\mathcal{H}^*}'(\mathbb{R}^n) :=\mathcal{L}(C_{\mathcal{H}}^{\infty}(\mathbb{R}^n), \mathbb{C} )$). Also, the space $ \mathcal{D}_{\mathcal{H}}^{\prime}(\mathbb{R}^n) \otimes \mathcal{S}^{\prime}(\mathcal{I})$ is the collection of all functions $f:\mathbb{R}^n \times \mathcal{I} \rightarrow \mathbb{C}$ such that $f(x,.) \in \mathcal{S}^{\prime}(\mathcal{I})$ for every $x \in \mathbb{R}^n$ and $f(.,\xi) \in \mathcal{D}_{\mathcal{H}}^{\prime}(\mathbb{R}^n)$ for every $\xi \in \mathcal{I}.$

\subsection{$\mathcal{H}$-Fourier transform}
The definition and some significant characteristics of the $\mathcal{H}$-Fourier transform, are reviewed in this subsection.

Let $\mathcal{S}(\mathcal{I})$ be the space of rapidly decaying functions from $\mathcal{I}$ to $\mathbb{C}$ and $\phi \in \mathcal{S}(\mathcal{I})$, i.e., for every $\ell<\infty,$ there exists a constant $C_{\phi, \ell}$ such that
$$
|\phi(\xi)| \leq C_{\phi, \ell}\langle\xi\rangle^{-\ell}
$$
for all $\xi \in \mathcal{I}$. The topology on $\mathcal{S}(\mathcal{I})$ is given by the seminorms $p_{k}, k \in \mathbb{N}_{0}$, i.e., 
$$
p_{k}(\phi):=\sup _{\xi \in \mathcal{I}}\langle\xi\rangle^{k}|\phi(\xi)|.
$$
Also, the continuous linear functionals on $\mathcal{S}(\mathcal{I})$ are of the form $$\phi \mapsto\langle u, \phi \rangle=\sum_{\xi \in \mathcal{I}} u(\xi) \phi(\xi),$$ where  $u$ is a function from $ \mathcal{I}$ to $\mathbb{C}$ with a property that it grows at most polynomially at infinity, i.e., there exist constants $\ell<\infty$ and $C_{u, \ell}$ such that 
$$
|u(\xi)| \leq C_{u, \ell}\langle\xi\rangle^{\ell}
$$
for all $\xi \in \mathcal{I}$. The collection of such $u: \mathcal{I} \rightarrow \mathbb{C}$ form a distribution space which is denoted by $\mathcal{S}^{\prime}(\mathcal{I})$.
\begin{defi}[Fourier transform]\label{fourier transform defi ch5}
	The $\mathcal{H}$-Fourier transform
	$$
	\left(\mathcal{F}_{\mathcal{H}} f\right)(\xi)=(f \mapsto \widehat{f}): C_{\mathcal{H}}^{\infty}(\mathbb{R}^n) \rightarrow \mathcal{S}(\mathcal{I})
	$$
	is given by
	$$
	\hat{f}(\xi):=\left(\mathcal{F}_{\mathcal{H}} f\right)(\xi)=\int_{\mathbb{R}^n} f(x) \overline{v_{\xi}(x)} ~d x.
	$$
\end{defi}
Similarly, the  $\mathcal{H}^{*}$-Fourier transform 
$$
\left(\mathcal{F}_{\mathcal{H} \cdot^{*}} f\right)(\xi)=\left(f \rightarrow \widehat{f}_{*}\right): C_{\mathcal{H}^{*}}^{\infty}(\mathbb{R}^n) \rightarrow S(\mathcal{I})
$$
is defined by
$$
\hat{f}_*(\xi):=\left(\mathcal{F}_{\mathcal{H}^{*}}  f\right)(\xi)=\int_{\mathbb{R}^n} f(x) \overline{u_{\xi}(x)} ~d x.
$$
\begin{prop}\label{inverse fourier transform defi ch5}
	The $\mathcal{H}$-Fourier transform $\mathcal{F}_{\mathcal{H}}$ is a bijective homeomorphism from $C_{\mathcal{H}}^{\infty}(\mathbb{R}^n)$ into $ \mathcal{S}(\mathcal{I})$. The inverse of $\mathcal{F}_{\mathcal{H}}$,
	$$\mathcal{F}_{\mathcal{H}}^{-1}: \mathcal{S}(\mathcal{I}) \rightarrow C_{\mathcal{H}}^{\infty}(\mathbb{R}^n)$$ is given by
	$$
	\left(\mathcal{F}_{\mathcal{H}}^{-1} h \right)(x)=\sum_{\xi \in \mathcal{I}} h(\xi) u_{\xi}(x),~ h \in \mathcal{S}(\mathcal{I})
	$$
	in order that the Fourier inversion formula is given by
	$$
	f(x)=\sum_{\xi \in \mathcal{I}} \hat{f}(\xi) u_\xi(x),~~ f \in C_{\mathcal{H}}^{\infty}(\mathbb{R}^n).
	$$
\end{prop}
Similarly,  $\mathcal{H}^*$-Fourier transform $\mathcal{F}_{\mathcal{H}^{*}}: C_{\mathcal{H}^{*}}^{\infty}(\mathbb{R}^n) \rightarrow \mathcal{S}(\mathcal{I})$ is a bijective homeomorphism and its inverse $\mathcal{F}_{\mathcal{H}^{*}}^{-1}: \mathcal{S}(\mathcal{I}) \rightarrow C_{\mathcal{H}^{*}}^{\infty}(\mathbb{R}^n)$ is
given by
$$
\left(\mathcal{F}_{\mathcal{H}^{*}}^{-1} h\right)(x)=\sum_{\xi \in \mathcal{I}} h(\xi) v_{\xi}(x), h \in S(\mathcal{I})
$$
such  that the conjugate Fourier inversion formula is given by
$$
f(x)=\sum_{\xi \in \mathcal{I}} \widehat{f}_{*}(\xi) v_{\xi}(x), f \in C_{\mathcal{H}^{*}}^{\infty}(\mathbb{R}^n).
$$
\subsection{Plancherel formula and Sobolev spaces} 
We revisit the Sobolev spaces generated by the operator $\mathcal{H}$ in this subsection. Additionally, we introduce appropriate sequence spaces $\ell^2(\mathcal{H})$ and $\ell^2(\mathcal{H}^*)$ to obtain the Plancherel identity.

\begin{defi}\label{l^p space defi ch5}
	Let $\ell^2(\mathcal{H})$ be the linear space of complex-valued functions $g$ on  $\mathcal{I}$ such that $\mathcal{F}_{\mathcal{H}}^{-1} g \in L^{2}(\mathbb{R}^n),$ i.e., if there exists $f \in L^{2}(\mathbb{R}^n)$ such that $\mathcal{F}_{\mathcal{H}} f=g.$ Then the space of sequences $\ell^{2}({\mathcal{H}})$ is a Hilbert space with the inner product
	$$
	\langle a, b\rangle_{\ell^{2}({\mathcal{H}})}:=\sum_{\xi \in \mathcal{I}} a(\xi) \overline{\left(\mathcal{F}_{\mathcal{H}^{*}} \circ \mathcal{F}_{\mathcal{H}}^{-1} b\right)(\xi)}
	$$
	for arbitrary $a, b \in \ell^{2}({\mathcal{H}}).$
\end{defi}
Here the norm on $\ell^{2}({\mathcal{H}})$ is given by  
$$
\|a\|_{\ell^{2}({\mathcal{H}})}=\sum_{\xi \in \mathcal{I}} a(\xi) \overline{\left(\mathcal{F}_{\mathcal{H}^{*}} \circ \mathcal{F}_{\mathcal{H}}^{-1} a\right)(\xi)}, ~~\quad a\in \ell^{2}({\mathcal{H}}).
$$
In the similar manner, the Hilbert space  $\ell^2(\mathcal{H}^*)$ be the linear space of complex-valued functions $g$ on  $\mathcal{I}$ such that $\mathcal{F}_{\mathcal{H}^*}^{-1} g \in L^{2}(\mathbb{R}^n),$  with the inner product
$$
\langle a, b\rangle_{\ell^{2}({\mathcal{H}^*})}:=\sum_{\xi \in \mathcal{I}} a(\xi) \overline{\left(\mathcal{F}_{\mathcal{H}} \circ \mathcal{F}_{\mathcal{H}^*}^{-1} b\right)(\xi)}
$$
for arbitrary $a, b \in \ell^2(\mathcal{H}^*)$.
The   sequence spaces $\ell^2(\mathcal{H})$ 
and  $\ell^2(\mathcal{H}^*)$  are thus generated by the biorthogonal
systems $\{u_\xi\}_{\xi \in \mathcal{I}}$ and $\{v_\xi\}_{\xi \in \mathcal{I}}$ respectively. The following Plancherel's identity provides a clear indication of the reasoning behind this specific definition selection.

\begin{prop}[Plancherel's identity]\label{plancherel's identity defi ch5}
	If $f_1, f_2 \in L^{2}(\mathbb{R}^n)$ then $\widehat{f_1}, \widehat{f_2} \in \ell^{2}({\mathcal{H}}), \widehat{f_1}_{*}, \widehat{f_2}_{*} \in \ell^{2}({\mathcal{H}^*})$
	and the inner products  take the form
	$$
	\langle\widehat{f_1}, \widehat{f_2} \rangle_{\ell^{2}({\mathcal{H}})}=\sum_{\xi \in \mathcal{I}} \widehat{f_1}(\xi) \overline{\widehat{f_2}_{*}(\xi)}
	$$
	and
	$$
	\langle\widehat{f_1}_{*}, \widehat{f_2}_{*} \rangle_{\ell^{2}({\mathcal{H}^*})}=\sum_{\xi \in \mathcal{I}} \widehat{f_1}_{*}(\xi) \overline{\widehat{f_2}(\xi)}.
	$$
	Particularly, we have
	$$
	\overline{\langle\widehat{f_1}, \widehat{f_2} \rangle}_{\ell^{2}({\mathcal{H}})}=\langle \widehat{f_2}_{*}, \widehat{f_1}_{*}\rangle_{\ell^{2}({\mathcal{H}^*})}
	$$
	and the Parseval identity takes of the form
	$$
	\langle {f_1},  {f_2} \rangle_{L^2(\mathbb{R}^n)}=\langle\widehat{f_1}, \widehat{f_2} \rangle_{\ell^{2}({\mathcal{H}})}=\sum_{\xi \in \mathcal{I}} \widehat{f_1}(\xi) \overline{\widehat{f_2}_{*}(\xi)}=\sum_{\xi \in \mathcal{I}} \widehat{f_1}_{*}(\xi) \overline{\widehat{f_2}(\xi)}.
	$$
	Also, for any $f \in L^{2}(\mathbb{R}^n),$ we have $\widehat{f} \in \ell^{2}({\mathcal{H}}), \widehat{f}_{*} \in \ell^{2}({\mathcal{H}^*})$ and
	$$
	\|f\|_{L^2(\mathbb{R}^n)}=\|\widehat{f}\|_{\ell^{2}({\mathcal{H}})}=\|\widehat{f}_{*}\|_{\ell^{2}({\mathcal{H}^*})}.
	$$
\end{prop}

Let us now recall the definition of Sobolev spaces generated by the operator $\mathcal{H}$.
\begin{defi}[Sobolev spaces $\mathrm{H}_{\mathcal{H}}^{s}(\mathbb{R}^n)$]\label{sobolev space defi ch5}
	For $f \in \mathcal{D}_{\mathcal{H}}^{\prime}(\mathbb{R}^n) \cap \mathcal{D}_{\mathcal{H}^{*}}^{\prime}(\mathbb{R}^n)$ and $s \in \mathbb{R},$ we
	say that $f \in \mathrm{H}_{\mathcal{H}}^{s}(\mathbb{R}^n)$ if and only if $\langle\xi\rangle^{s} \widehat{f}(\xi) \in  \ell^{2}({\mathcal{H}^*})$. The norm on $\mathrm{H}_{\mathcal{H}}^{s}(\mathbb{R}^n)$ is defined by
	$$
	\|f\|_{\mathrm{H}_{\mathcal{H}}^{s}(\mathbb{R}^n)}:=\left(\sum_{\xi \in \mathcal{I}}\langle\xi\rangle^{2 s} \widehat{f}(\xi) \overline{\widehat{f}_{*}(\xi)}\right)^{1/2}.
	$$
\end{defi}
The Sobolev space $\mathrm{H}_{ \mathcal{H}}^{s}(\mathbb{R}^n)$ is then the space of $\mathcal{H}$-distributions $f$ for which $\|f\|_{\mathrm{H}_{\mathcal{H}}^{s}(\mathbb{R}^n)}<\infty$ and  $\mathrm{H}_{\mathcal{H}^{*}}^{s}(\mathbb{R}^n)$  is the space of $\mathcal{H}^*$-distributions $f$ for which 
$$
\|f\|_{\mathrm{H}_{\mathcal{H}^{*}}^s(\mathbb{R}^n)}:=\left(\sum_{\xi \in \mathcal{I}}\langle\xi\rangle^{2 s} \widehat{f}_{*}(\xi) \overline{\widehat{f}(\xi)}\right)^{1 / 2}<\infty.
$$
Observe that $\mathrm{H}_{ \mathcal{H}}^{s}(\mathbb{R}^n)= \mathrm{H}_{ \mathcal{H}^*}^{s}(\mathbb{R}^n)$. For every $s \in \mathbb{R},$ the Sobolev space $\mathrm{H}_{\mathcal{H}}^{s}(\mathbb{R}^n)$ is a Hilbert space with the inner product given by 
$$
\langle f, g\rangle_{\mathrm{H}_{\mathcal{H}}^{s}(\mathbb{R}^n)}:=\sum_{\xi \in \mathcal{I}}\langle\xi\rangle^{2 s} \widehat{f}(\xi) \overline{\widehat{g}_{*}(\xi)}
$$
Similarly, with respect to the following inner product  
$$
\langle f, g\rangle_{\mathrm{H}_{\mathcal{H}^*}^{s}(\mathbb{R}^n)}:=\sum_{\xi \in \mathcal{I}}\langle\xi\rangle^{2 s} \widehat{f}_{*}(\xi) \overline{\widehat{g}(\xi)},
$$
the Sobolev space $\mathrm{H}_{\mathcal{H}^{*}}^{s}(\mathbb{R}^n)$ is also a Hilbert space for every $s \in \mathbb{R}$. It can be observed that the Sobolev spaces $ \mathrm{H}_{\mathcal{H}}^{s}(\mathbb{R}^n)$ and $\mathrm{H}_{\mathcal{H}^{*}}^{s}(\mathbb{R}^n)$ are isometrically isomorphic and in particular $\mathrm{H}_{\mathcal{H}}^{0}(\mathbb{R}^n)=L^2(\mathbb{R}^n)$. Also, for real $s\geq 1$, we define the (Roumieu) Gevrey space $\gamma_{\mathcal{H}}^s$ by the formula
\begin{equation}\label{gevrey space definition ch5}
     f\in \gamma_{\mathcal{H}}^s \Longleftrightarrow \exists A>0: \sum_{\xi \in \mathcal{I}} \mathrm{e}^{2 A \langle\xi\rangle^{\frac{1}{s}}}|\widehat{f}(\xi)|^2<\infty .
\end{equation}
\section{Proofs of the main results}\label{proof of results ch5}
In this section, our first aim is to prove the well-posedness to the
Cauchy problem \eqref{cauchy problem ch5} with regular coefficients. Here we  will consider only the $\mathcal{H}$-Fourier transform. However, one can similarly obtain all the results by considering $\mathcal{H}^{*}$-Fourier transform.

Taking the $\mathcal{H}$-Fourier transform of \eqref{cauchy problem ch5} with respect to  $x \in \mathbb{R}^n$, we obtain
	\begin{equation}\label{fourier transform of cauchy problem ch5}
	\left\{\begin{array}{l}
		\partial_{t}^{2} \widehat{v}(t, \xi)+ |\lambda_{\xi}| a(t) \widehat{v}(t, \xi)+q(t) \widehat{v}(t, \xi)=\widehat{f}(t, \xi), \quad(t, \xi) \in(0, T] \times \mathcal{I}, \\
		\widehat{v}(0, \xi)=\widehat{v}_{0}(\xi), \quad \xi \in \mathcal{I}, \\
		\partial_{t} \widehat{v}(0, \xi)=\widehat{v}_{1}(\xi), \quad \xi \in \mathcal{I}.
	\end{array}\right.
	\end{equation}
The main concept behind our further analysis is that we can analyze each equation in \eqref{fourier transform of cauchy problem ch5} individually and then using the Plancherel formula one can obtain required  estimates. We can reformulate  the Cauchy problem  \eqref{fourier transform of cauchy problem ch5} in the $\mathcal{H}$-Fourier space  as the first order system using the following matrix transformation:
\begin{equation}\label{matrix transformation ch5}
	V(t, \xi):=\left(\begin{array}{c}
		i\langle\xi\rangle \widehat{v}(t, \xi) \\
		\partial_{t} \widehat{v}(t, \xi)
	\end{array}\right), \quad V_{0}(\xi):=\left(\begin{array}{c}
		i\langle\xi\rangle \widehat{v}_{0}(\xi) \\
		\widehat{v}_{1}(\xi)
	\end{array}\right).
\end{equation}
Using above transformation, we see that \eqref{fourier transform of cauchy problem ch5} is equivalent to the following first order system:
\begin{equation}\label{equivalent cauchy problem ch5}
\left\{\begin{array}{l}
	\partial_{t} V(t, \xi)=i\langle\xi\rangle A(t) V(t, \xi)+i\langle\xi\rangle^{-1} Q(t) V(t, \xi)+F(t, \xi), \quad(t, \xi) \in(0, T] \times \mathcal{I}, \\
	V(0, \xi)=V_{0}(\xi), \quad \xi \in \mathcal{I},
\end{array}\right.
\end{equation}
where 
\begin{equation}\label{matrix A and Q ch5}
	A(t):=\left(\begin{array}{cc}
		0 & 1 \\
		a(t) & 0
	\end{array}\right), ~~ \quad Q(t):=\left(\begin{array}{cc}
		0 & 0 \\
		q(t)-a(t) & 0
	\end{array}\right), ~~ \text { and } F(t, \xi):=\left(\begin{array}{c}
0 \\
\widehat{f}(t, \xi)
\end{array}\right).
\end{equation}
Note that the matrix $A(t)$ has eigenvalues $\pm \sqrt{a(t)}$ and its symmetriser $S$ is given by
$$
S(t)=\left(\begin{array}{cc}
	a(t) & 0 \\
	0 & 1
\end{array}\right),
$$
i.e., we have
$$
S A-A^{*} S=0 .
$$
One can observe that $S(t)$ is positive for each $t\in [0,T]$, and 
\begin{equation}\label{symmetriser norm ch5}
	\|S(t)\| \leq(1+|a(t)|), \quad \text {for all } t \in[0, T].
\end{equation}  
Also from definition of $S$ and $Q$, we have
$$
\partial_{t} S(t):=\left(\begin{array}{cc}
	\partial_{t} a(t) & 0 \\
	0 & 0
\end{array}\right) \text { and }\left(S Q-Q^{*} S\right)(t):=\left(\begin{array}{cc}
	0 & a(t)-q(t) \\
	q(t)-a(t) & 0
\end{array}\right).
$$
Hence 
\begin{equation}\label{SQ-QS estimate ch5}
	\left\|\partial_{t}S(t)\right\| \leq\left|\partial_{t} a(t)\right|~~ \text{and}~~ \left\|\left(S Q-Q^{*} S\right)(t)\right\| \leq|q(t)|+|a(t)|, \quad \text{for all}~~ t\in[0, T].
\end{equation}
Now we are ready to prove the well-posedness results for Case 1-Case 4. We start with our first key  result in the scenario when the propagation speed $a$ is regular. 
\begin{proof}[Proof of Theorem \ref{classical sol case1 ch5} \textbf{(Case 1)}]
First, we define the energy
$$
E(t, \xi):=(S(t) V(t, \xi), V(t, \xi)).
$$
Then 
\begin{equation}\label{upper energy estimate ch5}
\begin{aligned}[b]
	E(t, \xi)=(S(t) V(t, \xi), V(t, \xi))& =\langle\xi\rangle^{2} a(t) |\widehat{v}(t, \xi)|^{2}+\left|\partial_{t} \widehat{v}(t, \xi)\right|^{2} \\
	& \leq \sup _{t \in[0, T]}\{a(t), 1\}\left(\langle\xi\rangle^{2}|\widehat{v}(t, \xi)|^{2}+\left|\partial_{t} \widehat{v}(t, \xi)\right|^{2}\right) \\
	& =c_{0}|V(t, \xi)|^{2}, \quad \text{for all} ~~t \in[0, T] ~~\text{and} ~~\xi \in \mathcal{I},
\end{aligned}
\end{equation}
where $c_{0} =\underset{t \in[0, T]}{ \sup} \{a(t), 1\}$. Similarly, we can deduce that 
\begin{equation}\label{lower energy estimate ch5}
		E(t, \xi) \geq c_{1}|V(t, \xi)|^{2}, \quad \text{for all} ~~t \in[0, T] ~~\text{and} ~~\xi \in \mathcal{I},
\end{equation}
where $c_{1} =\underset{t \in[0, T]}{ \inf} \{a(t), 1\}$. Using \eqref{symmetriser norm ch5} and \eqref{SQ-QS estimate ch5}, we obtain
\begin{equation}\label{derivative energy estimate ch5}
	\begin{aligned}[b]
		\partial_{t} E(t, \xi)= & \left(\partial_{t} S(t) V(t, \xi), V(t, \xi)\right)+\left(S(t) \partial_{t} V(t, \xi), V(t, \xi)\right)+\left(S(t) V(t, \xi), \partial_{t} V(t, \xi)\right) \\
		= & \left(\partial_{t} S(t) V(t, \xi), V(t, \xi)\right)+i\langle\xi\rangle(S(t) A(t) V(t, \xi), V(t, \xi))+ (S(t) F(t, \xi), V(t, \xi))\\
		&+ i\langle\xi\rangle^{-1}(S(t) Q(t) V(t, \xi), V(t, \xi))- i\langle\xi\rangle(S(t) V(t, \xi), A(t) V(t, \xi))\\
		& -i\langle\xi\rangle^{-1}(S(t) V(t, \xi), Q(t) V(t, \xi))+(S(t), F(t, \xi) V(t, \xi)) \\
		= & \left(\partial_{t} S(t) V(t, \xi), V(t, \xi)\right)+i\langle\xi\rangle\left(\left(S A-A^{*} S\right)(t) V(t, \xi), V(t, \xi)\right)+ \\
		& i\langle\xi\rangle^{-1}\left(\left(S Q-Q^{*} S\right)(t) V(t, \xi), V(t, \xi)\right)+2 \operatorname{Re}(S(t) F(t, \xi), V(t, \xi))\\
		= & \left(\partial_{t} S(t) V(t, \xi), V(t, \xi)\right)+i\langle\xi\rangle^{-1}\left(\left(S Q-Q^{*} S\right)(t) V(t, \xi), V(t, \xi)\right)+2 \operatorname{Re}(S(t) F(t, \xi), V(t, \xi))\\
		\leq& \left\|\partial_{t} S(t)\right\||V(t, \xi)|^{2}+\left\|\left(S Q-Q^{*} S\right)(t)\right\||V(t, \xi)|^{2}+2\|S(t)\||F(t, \xi) \| V(t, \xi)| \\
		\leq& \left(\left\|\partial_{t} S(t)\right\|+\left\|\left(S Q-Q^{*} S\right)(t)\right\|+\|S(t)\|\right)|V(t, \xi)|^{2}+\|S(t)\||F(t, \xi)|^{2}\\
		\leq& \left(1+\|\partial_{t}a\|_{L^{\infty}}+2\|a\|_{L^{\infty}}+\|q\|_{L^{\infty}}\right)|V(t, \xi)|^{2}+\left(1+\|a\|_{L^{\infty}}\right)|F(t, \xi)|^{2}\\
		=& c^{\prime}|V(t, \xi)|^{2}+ c^{\prime \prime} |F(t, \xi)|^{2},
	\end{aligned}
\end{equation}
where $c^{\prime} = 1+\|\partial_{t}a\|_{L^{\infty}}+2\|a\|_{L^{\infty}}+\|q\|_{L^{\infty}}$ and $c^{\prime \prime}=1+\|a\|_{L^{\infty}}$. Applying the Gronwall’s lemma to inequality \eqref{derivative energy estimate ch5}, we can conclude that
\begin{equation}\label{energy estimate relation ch5}
	E(t, \xi) \leq e^{\int_{0}^{t} c^{\prime} \mathrm{~d} \tau}\left(E(0, \xi)+\int_{0}^{t} c^{\prime \prime} |F(\tau, \xi)|^{2} \mathrm{~d} \tau\right),
\end{equation}
or all $t \in[0, T]$ and $\xi \in \mathcal{I}$. Now, combining \eqref{upper energy estimate ch5}, \eqref{lower energy estimate ch5}, and \eqref{energy estimate relation ch5}, we deduce that
$$
\begin{aligned}
c_{0}|V(t, \xi)|^{2} \leq E(t, \xi) &\leq e^{\int_{0}^{t} c^{\prime} \mathrm{~d} \tau}  \left(E(0, \xi)+\int_{0}^{t} c^{\prime \prime}|F(\tau, \xi)|^{2} \mathrm{~d} \tau\right) \\
& \leq e^{c^{\prime} T}\left(c_{1}|V(0, \xi)|^{2}+c^{\prime \prime} \int_{0}^{T}|F(\tau, \xi)|^{2} \mathrm{~d} \tau\right) .
\end{aligned}
$$
Therefore, for all $t \in[0, T]$ and  $\xi \in \mathcal{I}$, there exists a constant $C_{00}>0$ such that
\begin{equation}
	|V(t,\xi)|^2 \leq C_{00} \left(|V(0,\xi)|^2+\int_{0}^{T}|F(\tau, \xi)|^{2} \mathrm{~d} \tau\right).
\end{equation}
Hence
\begin{equation}\label{fourier coeffiecnt of v estimate ch5}
\begin{aligned}[b]
	\langle\xi\rangle^{2}|\widehat{v}(t, \xi)|^{2}+\left|\partial_{t} \widehat{v}(t, \xi)\right|^{2} &\leq C_{00}\left(\langle\xi\rangle^{2}|\widehat{v}(0, \xi)|^{2}+\left|\partial_{t} \widehat{v}(0, \xi)\right|^{2}+\int_{0}^{T}|F(\tau, \xi)|^{2} \mathrm{~d} \tau\right)\\
 & \leq C_{00}\left(\langle\xi\rangle^{2}|\widehat{v}(0, \xi)|^{2}+\left|\partial_{t} \widehat{v}(0, \xi)\right|^{2}+\int_{0}^{T}|\widehat{f}(\tau, \xi)|^{2} \mathrm{~d} \tau\right).
\end{aligned}
\end{equation}
Multiplying \eqref{fourier coeffiecnt of v estimate ch5} with $\langle\xi\rangle^{2s}$ and   applying Plancherel’s formula, we can conclude that
$$
	\|v(t, \cdot)\|_{\mathrm{H}_{\mathcal{H}}^{1+s}}^{2}+\left\|v_{t}(t, \cdot)\right\|_{\mathrm{H}_{\mathcal{H}}^{s}}^{2} \leq C\left(\left\|v_{0}\right\|_{\mathrm{H}_{\mathcal{H}}^{1+s}}^{2}+\left\|v_{1}\right\|_{\mathrm{H}_{\mathcal{H}}^{s}}^{2}+\|f\|_{L^{2}\left([0, T] ; \mathrm{H}_{\mathcal{H}}^{s}\right)}^{2}\right),
$$
for all $t \in[0, T]$, where the constant $C>0$ independent of $t \in[0, T]$. This completes the proof of the theorem.
\end{proof}
Now we prove our second  key result in the scenario when the propagation speed  $a  $  is a H\"{o}lder continuous function  of order $0<\alpha<1.$ \begin{proof}[Proof of Theorem \ref{classical sol case2 ch5} \textbf{(Case 2)}]
We assume that  $a \in \mathcal{C}^\alpha([0, T])$ with $0<\alpha<1$   such that $\underset{t \in [o,T]}{\inf} a(t)=a_{0}>0.$ In this case, we apply the technique established by Colombini and Kinoshita in \cite{colombini kinoshita}. Our aim is to find the solution of the system \eqref{equivalent cauchy problem ch5} in the following form:
\begin{equation}\label{solution form in case 2 ch5}
    V(t, \xi)=e^{-\rho(t) \langle \xi \rangle^{\frac{1}{s}}}(\operatorname{det} H(t))^{-1} H(t) W(t,\xi),
\end{equation}
where $\rho \in C^{1}([0, T])$ is a real-valued function with the property that $\rho(0)=0$, $W=W(t, \xi)$ will be determined later and 
$$
H(t)=\left(\begin{array}{cc}
1 & 1 \\
-\lambda^{\varepsilon}(t) & \lambda^{\varepsilon}(t)
\end{array}\right).
$$
For all $\varepsilon \in (0,1]$, here we define $\lambda^{\varepsilon}(t)$  as $$\lambda^{\varepsilon}(t)=\left(\sqrt{a}*\psi_{\varepsilon}\right)(t)$$
with $\psi_{\varepsilon}(t)=\frac{1}{\varepsilon}\psi(\frac{t}{\varepsilon})$ for each $t\in[0,T]$, and  $\psi \in C_{c}^{\infty}(\mathbb{R})$ be a non-negative function such that $$supp(\psi)\subseteq [0,T]\quad \text{and}\quad  \int_{0}^{T} \psi(x) \mathrm{~d} x =1.$$  By above construction, we can conclude that $\lambda^{\varepsilon}\in C^{\infty}([0,T])$ and 
$ \lambda^{\varepsilon}(t)\geq \sqrt{a_{0}}$ 
for all $t \in[0, T]$ and $\varepsilon \in(0,1].$ Therefore,
$ \operatorname{det} H(t)=2\lambda^{\varepsilon}(t)\geq 2 \sqrt{a_{0}}.$ 
Furthermore, for every $t\in [0, T]$,  using the H\"{o}lder regularity property of $a(t)$ with an order of $\alpha$, we can deduce that  
\begin{equation}\label{estimate for a-lambda ch5}
\begin{aligned}[b]
\left|\lambda^\varepsilon(t)-\sqrt{a(t)}\right| & =\left|\left(\sqrt{a} * \psi_\varepsilon\right)(t)-\sqrt{a(t)}\right| \\
& =\left|\int_{0}^{t} \sqrt{a(t-y)} \psi_\varepsilon(y) \mathrm{~d} y-\sqrt{a(t)} \int_{0}^{T} \psi(y) \mathrm{~d} y\right| \\
& =\left|\int_{0}^{\frac{t}{\varepsilon}} \sqrt{a(t-\varepsilon x)} \psi(x) \mathrm{~d} x-\sqrt{a(t)} \int_{0}^{T} \psi(x) \mathrm{~d} x\right| \\
& \leq\int_{0}^{\frac{t}{\varepsilon}} \frac{\mid a(t-\varepsilon x)-a(t)) \mid}{\sqrt{a(t-\varepsilon x)}-\sqrt{a(t)}} \psi(x) \mathrm{~d} x \\
& \leq \frac{\|a\|_{\mathcal{C}^\alpha([0, T])}}{2 \sqrt{a_0}} \varepsilon^\alpha .
\end{aligned}
\end{equation}
Now, substituting \eqref{solution form in case 2 ch5} in \eqref{equivalent cauchy problem ch5}, we have
$$
\begin{aligned}
&e^{-\rho(t) \langle \xi \rangle^{\frac{1}{s}}}(\operatorname{det} H(t))^{-1} H(t) \partial_{t} W(t,\xi) +e^{-\rho(t) \langle \xi \rangle^{\frac{1}{s}}(\xi)}\left(-\rho^{\prime}(t) \langle \xi \rangle^{\frac{1}{s}}\right)(\operatorname{det} H(t))^{-1} H(t) W(t,\xi) \\
& -e^{-\rho(t) \langle \xi \rangle^{\frac{1}{s}}} \frac{\partial_{t} \operatorname{det} H(t)}{(\operatorname{det} H(t))^{2}} H(t) W(t,\xi) +e^{-\rho(t) \langle \xi \rangle^{\frac{1}{s}}}(\operatorname{det} H(t))^{-1}\left(\partial_{t} H(t)\right) W(t,\xi) \\
& =i \langle \xi \rangle e^{-\rho(t) \langle \xi \rangle^{\frac{1}{s}}}(\operatorname{det} H(t))^{-1} A(t) H(t) W(t,\xi)+i\langle\xi\rangle^{-1} e^{-\rho(t) \langle \xi \rangle^{\frac{1}{s}}} (\operatorname{det} H(t))^{-1} Q(t)H(t) W(t,\xi) + F(t,\xi).
\end{aligned}
$$
Multiplying the above equation by $e^{\rho(t) \langle \xi \rangle^{\frac{1}{s}}} (\operatorname{det} H(t)) H(t)^{-1}$, we get
$$
\begin{aligned}
&\partial_{t} W(t,\xi) -\rho^{\prime}(t) \langle \xi \rangle^{\frac{1}{s}} W(t,\xi)- \frac{\partial_{t} \operatorname{det} H(t)}{\operatorname{det} H(t)}  W(t,\xi) +H(t)^{-1}\partial_{t} H(t) W(t,\xi) \\
& =i \langle \xi \rangle ( H(t))^{-1} A(t) H(t) W(t,\xi)+i\langle\xi\rangle^{-1} H(t)^{-1} Q(t)H(t) W(t,\xi)+e^{\rho(t) \langle \xi \rangle^{\frac{1}{s}}} (\operatorname{det} H(t)) H(t)^{-1} F(t,\xi).
\end{aligned}
$$
Therefore
\begin{equation}\label{derivate W(t) estimate ch5}
\begin{aligned}[b]
& \partial_{t}|W(t, \xi)|^{2}= 2 \operatorname{Re}\left(\partial_{t} W(t, \xi), W(t, \xi)\right) \\
&=2 \rho^{\prime}(t) \langle \xi \rangle^{\frac{1}{s}}|W(t, \xi)|^{2} +2 \frac{\partial_{t} \operatorname{det} H(t)}{\operatorname{det} H(t)}|W(t, \xi)|^{2} -2 \operatorname{Re}\left(H(t)^{-1} \partial_{t} H(t) W(t, \xi), W(t, \xi)\right) \\
&-2 \langle \xi \rangle \operatorname{Im}\left(H(t)^{-1} A(t) H(t) W(t, \xi), W(t, \xi)\right)-2 \langle \xi \rangle^{-1}\operatorname{Im}\left(H(t)^{-1} Q(t) H(t) W(t, \xi), W(t, \xi)\right) \\
&+2 e^{\rho(t) \langle \xi \rangle^{\frac{1}{s}}} \operatorname{Re}\left((\operatorname{det} H(t)) H(t)^{-1} F(t, \xi), W(t, \xi)\right).
\end{aligned}
\end{equation}
For any matrix $C=C(t) \in M_{2}(\mathbb{R})$, we know that
$$
2\operatorname{Im}\left(H(t)^{-1} C(t) H(t) W(t, \xi), W(t, \xi)\right) \leq \left\|H^{-1} C H-\left(H^{-1} C H\right)^{*}\right\||W(t, \xi)|^{2}
$$
 for each $t \in [0,T]$. This implies that
\begin{equation} \label{derivate of W(t) estimate 2 ch5}
\begin{aligned}[b]
\partial_{t}|W(t, \xi)|^{2}& \leq 2 \rho^{\prime}(t) \langle \xi \rangle^{\frac{1}{s}}|W(t, \xi)|^{2} +2\left|\frac{\partial_{t} \operatorname{det} H(t)}{\operatorname{det} H(t)}\right||W(t, \xi)|^{2} +2\left\|H^{-1} \partial_{t} H\right\|\left|W(t, \xi)^{2}\right|\\
& +\langle \xi \rangle\left\|H^{-1} A H-\left(H^{-1} A H\right)^{*}\right\||W(t, \xi)|^{2}+\langle \xi \rangle^{-1}\left\|H^{-1} Q H-\left(H^{-1} Q H\right)^{*}\right\||W(t, \xi)|^{2} \\
& +2 e^{\rho(t) \langle \xi \rangle^{\frac{1}{s}}}\left\|(\operatorname{det} H) H^{-1}\right\||F(t, \xi)||W(t, \xi)|.
\end{aligned}
\end{equation}
Now,  using the methods described in  \cite{garetto micheal,ruzhansky taranto}, we  prove the following   estimations to handle the aforementioned term:
\begin{enumerate}
    \item $\|\frac{\partial_{t} \operatorname{det} H}{\operatorname{det} H}\| \leq c_{1} \varepsilon^{\alpha-1}$,\\
    \item $\left\|H^{-1} \partial_{t} H\right\| \leq c_{2} \varepsilon^{\alpha-1}$,\\
    \item $\left\|H^{-1} A H-\left(H^{-1} A H\right)^{*}\right\| \leq c_{3} \varepsilon^{\alpha}$,\\
    \item  $\left\|H^{-1} Q H-\left(H^{-1} Q H\right)^{*}\right\| \leq c_{4} \varepsilon^{\alpha}$,\\
    \item $\left\|(\operatorname{det} H) H^{-1}\right\|\leq c_{5} \varepsilon^{\alpha},$
\end{enumerate}
where $c_{1},c_{2},c_{3},c_{4},c_{5}$ are positive constants. Proof of the above five estimates are as follows:
\begin{enumerate}
\item It is easy to check that 
$$\left|\frac{\partial_{t} \operatorname{det} H(t)}{\operatorname{det} H(t)}\right| = \left|\frac{2\partial_{t}\lambda_{\varepsilon}(t)}{2\lambda_{\varepsilon}(t)}\right|\leq \frac{|\partial_{t}\lambda_{\varepsilon}(t)|}{2\sqrt{a_{0}}}$$ and 
\begin{equation}\label{derivate bound of lambda ch5}
    \begin{aligned}[b]
    \left|\partial_{t} \lambda^{\varepsilon}(t)\right| & =\left|\sqrt{a} * \partial_{t} \psi_{\varepsilon}(t)\right|=\left|\frac{1}{\varepsilon} \int_{0}^{\frac{t}{\varepsilon}} \sqrt{a(t-s \varepsilon)} \psi^{\prime}(s) \mathrm{~d} s\right| \\
& \leq \frac{1}{\varepsilon} \int_{0}^{\frac{t}{\varepsilon}}|\sqrt{a(t-s \varepsilon)}-\sqrt{a(t)}| \psi^{\prime}(s) \mathrm{~d} s \\
& =\frac{1}{\varepsilon} \int_{0}^{\frac{t}{\varepsilon}} \frac{|a(t-s \varepsilon)-a(t)|}{\sqrt{a(t-s \varepsilon)}+\sqrt{a(t)}} \psi^{\prime}(s) \mathrm{~d} s \\
& \leq c^{\prime}_{1} \varepsilon^{\alpha-1}, \quad \text{for each $t\in[0,T]$},
    \end{aligned}
\end{equation}
where $c^{\prime}_{1}=\frac{\|a\|_{\mathcal{C}^\alpha([0, T])}}{2 \sqrt{a_0}} \int_{0}^{T}\psi^{\prime}(s) \mathrm{~d} s$. Thus
$$\left\|\frac{\partial_{t} \operatorname{det} H}{\operatorname{det} H}\right\| \leq \frac{c^{\prime}_{1} \varepsilon^{\alpha-1}}{2\sqrt{a_{0}}}=c_{1}\varepsilon^{\alpha-1} \quad \text{with  $c_{1}=\frac{c^{\prime}_{1}}{2\sqrt{a_{0}}}.$}$$
 
\item  
Note that $H^{-1}(t)=\frac{1}{2\lambda^{\varepsilon}(t)}\left(\begin{array}{cc}\lambda^{\varepsilon}(t) & -1 \\\lambda^{\varepsilon}(t) & 1\end{array}\right)$ and $\partial_{t}H(t)=\left(\begin{array}{cc}0 & 0 \\ \partial_{t} \lambda^{\varepsilon}(t) & \partial_{t} \lambda^{\varepsilon}(t)\end{array}\right)$. Thus   
 $$ H^{-1}(t)\partial_{t}H(t)=\frac{1}{2\lambda^{\varepsilon}(t)}\left(\begin{array}{cc}-\partial_{t}\lambda^{\varepsilon}(t) & -\partial_{t}\lambda^{\varepsilon}(t)\\\partial_{t}\lambda^{\varepsilon}(t) & \partial_{t}\lambda^{\varepsilon}(t)\end{array}\right).$$
Therefore
$$\left\|H^{-1} \partial_{t} H\right\| \leq \frac{c^{\prime}_{1} \varepsilon^{\alpha-1}}{2\sqrt{a_{0}}}=c_{2}\varepsilon^{\alpha-1} \quad \text{with $c_{2}=\frac{c^{\prime}_{1}}{2\sqrt{a_{0}}}.$}$$
 \item It is easy to verify that 
\begin{equation}\label{experession for H inverse AH ch5}
    H^{-1}(t) A(t) H(t)-\left(H^{-1}(t) A(t) H(t)\right)^{*}=\frac{1}{2\lambda^{\varepsilon}(t)}\left(\begin{array}{cc}
0 & -2 a(t)+2\left(\lambda^{\varepsilon}(t)\right)^{2}\\
2 a(t)-2\left(\lambda^{\varepsilon}(t)\right)^{2} & 0
\end{array}\right).
\end{equation}
Also,   from \eqref{estimate for a-lambda ch5}, one can deduce that 
\begin{equation}\label{estimate for a-lambda^2 ch5}
    \begin{aligned}[b]
        \left|a(t)-\left(\lambda^{\varepsilon}(t)\right)^{2}\right| & =\left|\left(\sqrt{a(t)}-\left(\lambda^{\varepsilon}(t)\right)\right)\left(\sqrt{a(t)}+\left(\lambda^{\varepsilon}(t)\right)\right)\right| \\
    & \leq \frac{\|a\|_{\mathcal{C}^\alpha([0, T])}}{2 \sqrt{a_0}} \varepsilon^{\alpha}\left|\sqrt{a(t)}+\int_{0}^{t} \sqrt{a(t-s)} \psi_{\varepsilon}(s) \mathrm{~d} s\right| \\
    & \leq \frac{\|a\|_{\mathcal{C}^\alpha([0, T])}}{\sqrt{a_0}}\|\sqrt{a}\|_{L^{\infty}([0, T])} \varepsilon^{\alpha}\\
    &=c_{3}^{\prime} \varepsilon^{\alpha}, \quad \text{where $c_{3}^{\prime}=\frac{\|a\|_{\mathcal{C}^\alpha([0, T])}}{\sqrt{a_0}}\|\sqrt{a}\|_{L^{\infty}([0, T])}.$}
    \end{aligned}
\end{equation}
    Thus
$$\left\|H^{-1} A H-\left(H^{-1} A H\right)^{*}\right\| \leq \frac{c_{3}^{\prime}\varepsilon^{\alpha}}{2\sqrt{a_{0}}} = c_{3} \varepsilon^{\alpha},\quad \text{
where} \quad c_{3}=\frac{c_{3}^{\prime}}{2\sqrt{a_{0}}}.$$
\item A direct calculations yields
\begin{equation}\label{experession for H inverse QH ch5}
    H^{-1}(t) Q(t) H(t)-\left(H^{-1}(t) Q(t) H(t)\right)^{*}=\frac{1}{2\lambda^{\varepsilon}(t)}\left(\begin{array}{cc}
-(q(t)-a(t)) & -(q(t)-a(t))\\
q(t)-a(t) & q(t)-a(t)
\end{array}\right).
\end{equation}
Therefore
$$\left\|H^{-1} Q H-\left(H^{-1} Q H\right)^{*}\right\| \leq \frac{\|a\|_{L^{\infty}([0, T])}+\|q\|_{L^{\infty}([0, T])}}{2\sqrt{a_{0}}}\varepsilon^{\alpha} =c_{4} \varepsilon^{\alpha},$$
where $c_{4}=\frac{\|a\|_{L^{\infty}([0, T])}+\|q\|_{L^{\infty}([0, T])}}{2\sqrt{a_{0}}}.$
\item  For each $t\in[0,T],$ a simple computation gives us 
$ |\lambda^{\varepsilon}(t)|\leq c_{5} \varepsilon^{\alpha}$ 
for some positive constant $c_{5}$. Thus
$$\left\|(\operatorname{det} H) H^{-1}\right\|=\left\|\lambda^{\varepsilon}\right\|\leq c_{5} \varepsilon^{\alpha}.$$
\end{enumerate}
Now,  using all the five estimates described above and    the fact that $|F(t, \xi)|\leq c_{0} e^{-A\langle \xi \rangle^{\frac{1}{s}}}$ for some constants $c_{0},A>0$,  the estimate    \eqref{derivate of W(t) estimate 2 ch5} further  reduces to 
\begin{equation} \label{modified derivate of W(t) estimate 2 ch5}
\begin{aligned}[b]
\partial_{t}|W(t, \xi)|^{2}& \leq\left(2 \rho^{\prime}(t) \langle \xi \rangle^{\frac{1}{s}} +2c_{1} \varepsilon^{\alpha-1}+2c_{2} \varepsilon^{\alpha-1} +\langle \xi \rangle c_{3} \varepsilon^{\alpha}+\langle \xi \rangle^{-1}c_{4} \varepsilon^{\alpha}\right)|W(t, \xi)|^{2} \\
& +2c_{0}c_{5} e^{\left(\rho(t)-A\right) \langle \xi \rangle^{\frac{1}{s}}} \varepsilon^{\alpha}|W(t, \xi)|.
\end{aligned}
\end{equation}
Let us choose $\varepsilon=\langle\xi\rangle^{-1}$ and define $\rho(t)=-kt,$ where $k>0$ will be chosen later. Then  \eqref{modified derivate of W(t) estimate 2 ch5} becomes
\begin{equation} \label{final modified derivate of W(t) estimate 2 ch5}
\begin{aligned}[b]
\partial_{t}|W(t, \xi)|^{2}& \leq\left(-2k \langle \xi \rangle^{\frac{1}{s}} +2c_{1} \langle\xi\rangle^{1-\alpha}+2c_{2} \langle\xi\rangle^{1-\alpha} +\langle \xi \rangle c_{3} \langle\xi\rangle^{-\alpha}+\langle \xi \rangle^{-1}c_{4} \langle\xi\rangle^{-\alpha}\right)|W(t, \xi)|^{2} \\
& +2c_{0}c_{5} e^{\left(-kt-A\right) \langle \xi \rangle^{\frac{1}{s}}} \langle\xi\rangle^{-\alpha}|W(t, \xi)|\\
 &\leq\left(-2k \langle \xi \rangle^{\frac{1}{s}} +2c_{1} \langle\xi\rangle^{1-\alpha}+2c_{2} \langle\xi\rangle^{1-\alpha} + c_{3} \langle\xi\rangle^{1-\alpha}+c_{4} \langle\xi\rangle^{1-\alpha}\right)|W(t, \xi)|^{2} \\
& +2c_{0}c_{5} e^{-A \langle \xi \rangle^{\frac{1}{s}}} \langle\xi\rangle^{-\alpha}|W(t, \xi)|.
\end{aligned}
\end{equation}
Without any loss of generality, we assume that $|W(t, \xi)|\geq1.$ We also  assume that $\frac{1}{s}>1-\alpha.$ Then
$$\partial_{t}|W(t, \xi)|^{2} \leq \left(-2k \langle \xi \rangle^{1-\alpha} +k_{0} \langle\xi\rangle^{1-\alpha}\right)|W(t, \xi)|^{2},$$
where $k_{0}=2c_1+c_2+c_3+c_4+2c_0c_5.$ Further, if we choose $k>\frac{k_{0}}{2}$, then for each $\xi \in \mathcal{I},$ we obtained that
$$\partial_{t}|W(t, \xi)|^{2} \leq 0.$$
This implies that $|W(t,\xi)|$ is a monotone function for each $\xi \in \mathcal{I},$ and hence
\begin{equation}\label{estimate of norm of V in case 2ch5}
    \begin{aligned}[b]
        |V(t,\xi)| & \leq\|H(t)\||\operatorname{det}(H(t))|^{-1} e^{k t\langle\xi\rangle^{\frac{1}{s}}}|W(t,\xi)| \\
        & \leq\|H(t)\||\operatorname{det} H(t)|^{-1} e^{kt\langle\xi\rangle^{\frac{1}{s}}}\|H(0)\|^{-1}|\operatorname{det} H(0)||V(0,\xi)|.
    \end{aligned}
\end{equation}
Since
 $\|H(t)\||\operatorname{det} H(t)|^{-1} \|H(0)\|^{-1}|\operatorname{det} H(0)|\leq b_{0}$ 
for some $b_{0}>0,$ we have
$$|V(t,\xi)|  \leq b_{0}e^{kT\langle\xi\rangle^{\frac{1}{s}}}|V(0,\xi)|.$$
Now, using \eqref{matrix transformation ch5}, we get
$$\langle\xi\rangle^{2}|\widehat{v}(t, \xi)|^{2}+\left|\partial_{t} \widehat{v}(t, \xi)\right|^{2} \leq b_{0}^2 e^{2kT\langle\xi\rangle^{\frac{1}{s}}}\left(\langle\xi\rangle^{2}|\widehat{v}(0, \xi)|^{2}+\left|\partial_{t} \widehat{v}(0, \xi)\right|^{2}\right).$$
Since the initial Cauchy data $v_{0}, v_{1} \in \gamma_{\mathcal{H}}^{s}$,   we can conclude that the Cauchy problem \eqref{cauchy problem ch5} has a unique solution $v \in C^2\left([0, T] ; \gamma_{\mathcal{H}}^{s}\right)$ provided  
 $$1\leq s < 1+ \frac{\alpha}{1-\alpha}.$$
\end{proof}
Now we prove our third  key result in the scenario when the propagation speed   $a \in \mathcal{C}^l([0, T])$ with $l\geq2.$ 
\begin{proof}[Proof of Theorem \ref{classical sol case3 ch5} \textbf{(Case 3)}]
In this case, we assume that $a \in \mathcal{C}^l([0, T])$ with $l\geq2,$   such that $a(t)\geq0.$ Using the notations from the preceding cases, we want to investigate the well-posedness of the system \eqref{equivalent cauchy problem ch5}. Let $V(t,\xi)$ be a column vector with column entry $V_{1}(t,\xi)$ and $V_{2}(t,\xi).$  Consider the quasi-symmetriser  
$$
	P_{\varepsilon}(t):=\left(\begin{array}{cc}
		a(t) & 0 \\
		0 & 1
	\end{array}\right)+\varepsilon^{2}\left(\begin{array}{ll}
		1 & 0 \\
		0 & 0
	\end{array}\right)
	$$ for the matrix $A(t)$ given in \eqref{matrix A and Q ch5}.
Then for each $t\in[0,T],$ we have the following estimate  
\begin{equation}\label{PV,V case3 ch5}
    \left(P_{\varepsilon}(t) V(t,\xi), V(t,\xi)\right)=\left(a(t)+\varepsilon^{2}\right)\left|V_{1}(t,\xi)\right|^{2}+\left|V_{2}(t,\xi)\right|^{2}.
\end{equation}
Keeping in mind 
\begin{equation}\label{PA-AP case3 ch5}
    P_{\varepsilon}(t) A(t)-A^{*}(t) P_{\varepsilon}(t)=\left(\begin{array}{cc}
			0 & \varepsilon^{2} \\
			-\varepsilon^{2} & 0
   \end{array}\right)
\end{equation}
for each $t\in[0,T],$  we can write  
\begin{equation}\label{PA-AP norm case3 ch5}
    \begin{aligned}[b]
        &i\left(\left(P_{\varepsilon}(t) A(t)-A^{*}(t) P_{\varepsilon}(t)\right) V(t, \xi), V(t, \xi)\right) \\
        &=i \left(\varepsilon^{2} \overline{V_{1}(t, \xi)} V_{2}(t, \xi)-\varepsilon^{2} V_{1}(t, \xi) \overline{V_{2}(t, \xi)}\right)\\
        &= -2i^2 \varepsilon^{2} \operatorname{Im}\left(V_{1}(t, \xi) \overline{V_{2}(t, \xi)}\right)\\
        &\leq 2\varepsilon^{2}|V_{1}(t, \xi)||V_{2}(t, \xi)|\\
        & \leq \varepsilon\left(\varepsilon^{2}\left|V_{1}(t, \xi)\right|^{2}+\left|V_{2}(t, \xi)\right|^{2}\right) \\
        & \leq \varepsilon\left(\left(a(t)+\varepsilon^{2}\right)\left|V_{1}(t, \xi)\right|^{2}+\left|V_{2}(t, \xi)\right|^{2}\right)\\
        & = \varepsilon\left(P_{\varepsilon}(t) V(t,\xi), V(t,\xi)\right).
    \end{aligned}
\end{equation}
Following the techniques given  in \cite[Theorem 6]{garetto micheal 2}, we can obtain the following estimate:
\begin{equation}\label{PQ-QP estimate bound case3 ch5}
    \left|i\left(\left(P_{\varepsilon}(t) Q(t)-Q^{*}(t) P_{\varepsilon}(t)\right) V(t, \xi), V(t, \xi)\right)\right| \leq c_{1} (P_{\varepsilon}(t) V(t, \xi), V(t, \xi))
\end{equation}
for some positive constant $c_{1}.$ Now, we define the energy as
$$
E_{\varepsilon}(t, \xi):=(P_{\varepsilon}(t) V(t, \xi), V(t, \xi)).
$$
Then 
\begin{equation}\label{derivative energy estimate case3 ch5}
	\begin{aligned}[b]
		\partial_{t} E_{\varepsilon}(t, \xi)= & \left(\partial_{t} P_{\varepsilon}(t) V(t, \xi), V(t, \xi)\right)+\left(P_{\varepsilon}(t) \partial_{t} V(t, \xi), V(t, \xi)\right)+\left(P_{\varepsilon}(t) V(t, \xi), \partial_{t} V(t, \xi)\right) \\
		= & \left(\partial_{t} P_{\varepsilon}(t) V(t, \xi), V(t, \xi)\right)+i\langle\xi\rangle\left(P_{\varepsilon}(t) A(t) V(t, \xi), V(t, \xi)\right)\\
         & +i\langle\xi\rangle^{-1}\left(P_{\varepsilon}(t) Q(t) V(t, \xi), V(t, \xi)\right)+\left(P_{\varepsilon}(t) F(t, \xi), V(t, \xi)\right)\\
         & -i\langle\xi\rangle\left(P_{\varepsilon}(t) V(t, \xi), A(t) V(t, \xi)\right)-i\langle\xi\rangle^{-1}\left(P_{\varepsilon}(t)  V(t, \xi), Q(t)V(t, \xi)\right)\\
         & +\left(P_{\varepsilon}(t) V(t, \xi), F(t, \xi)\right)\\
         =& \left(\partial_{t} P_{\varepsilon}(t) V(t, \xi), V(t, \xi)\right)+i \langle\xi\rangle\left(\left(P_{\varepsilon}(t) A(t)-A^{*}(t) P_{\varepsilon}(t)\right) V(t, \xi), V(t, \xi)\right)\\
         & +i \langle\xi\rangle^{-1}\left(\left(P_{\varepsilon}(t) Q(t)-Q^{*}(t) P_{\varepsilon}(t)\right) V(t, \xi), V(t, \xi)\right)\\
         &+ \left(P_{\varepsilon} F(t, \xi), V(t, \xi)\right)+\left(P_{\varepsilon} V(t, \xi), F(t, \xi)\right).
	\end{aligned}
\end{equation}
 For each $\varepsilon\in (0,1],$ using the equlity \eqref{PV,V case3 ch5},  we have
\begin{equation}\label{P,V right side estimate ch5}
    \begin{aligned}[b]
        \left(P_{\varepsilon}(t) V(t,\xi), V(t,\xi)\right)&=\left(a(t)+\varepsilon^{2}\right)\left|V_{1}(t,\xi)\right|^{2}+\left|V_{2}(t,\xi)\right|^{2}\\
        &\leq \left(\|a\|_{L^{\infty}[0,T]}+1\right)\left|V(t,\xi)\right|^{2}
    \end{aligned}
\end{equation}
and
\begin{equation}\label{P,V left side estimate ch5}
    \begin{aligned}[b]
        \left(P_{\varepsilon}(t) V(t,\xi), V(t,\xi)\right)&=\left(a(t)+\varepsilon^{2}\right)\left|V_{1}(t,\xi)\right|^{2}+\left|V_{2}(t,\xi)\right|^{2}\\
        & \geq \epsilon^{2}\left|V_{1}(t,\xi)\right|^{2}+\frac{\epsilon^{2}}{\|a\|_{L^{\infty}([0, T])}+1}\left|V_{2}(t,\xi)\right|^{2} \\
        & \geq \frac{\epsilon^{2}}{\|a\|_{L^{\infty}([0, T])}+1}\left|V_{1}(t,\xi)\right|^{2}+\frac{\epsilon^{2}}{\|a\|_{L^{\infty}([0, T])}+1}\left|V_{2}(t,\xi)\right|^{2} \\
        & =c_{2}^{-1} \epsilon^{2}|V(t,\xi)|^{2},
    \end{aligned}
\end{equation}
where $c_{2}=\|a\|_{L^{\infty}([0, T])}+1.$  Combining \eqref{P,V right side estimate ch5} and \eqref{P,V left side estimate ch5}, we have
\begin{equation}\label{E epsilon both side estimate case3 ch5}
    c_{2}^{-1} \epsilon^{2}|V(t,\xi)|^{2} \leq  \left(P_{\varepsilon}(t) V(t,\xi), V(t,\xi)\right)\leq c_{2}\left|V(t,\xi)\right|^{2}.
\end{equation}
Also, one can observe that
\begin{equation}\label{P,V left side estimate with V2 case3 ch5}
 \left(P_{\varepsilon}(t) V(t,\xi), V(t,\xi)\right)=\left(a(t)+\varepsilon^{2}\right)\left|V_{1}(t,\xi)\right|^{2}+\left|V_{2}(t,\xi)\right|^{2} \geq \left|V_{2}(t,\xi)\right|^{2}.
\end{equation}
 Since $f \in C\left([0, T] ; \gamma_{\mathcal{H}}^{s}\right)$, there exist positive constants  $c_{3} $ and $A$ such that $|\widehat{f}(t, \xi)| \leq c_{3} e^{-A \langle \xi \rangle^{\frac{1}{s}}}$ for all $t \in[0, T]$ and $\xi \in \mathcal{I}$.  Therefore 
 $$	\begin{aligned}
	    &\left(P_{\varepsilon}(t, \xi) F(t, \xi), V(t, \xi)\right)+\left(P_{\varepsilon}(t, \xi) V(t, \xi), F(t, \xi)\right)\\
     &= V_{2}(t, \xi)\overline{\widehat{f}(t, \xi)}+\widehat{f}(t, \xi)\overline{V_{2}(t, \xi)}\\
     &= 2\operatorname{Re}\left(V_{2}(t, \xi)\overline{\widehat{f}(t, \xi)}\right)\\
     &\leq 2|\widehat{f}(t, \xi)||V_{2}(t, \xi)|\\
     &\leq 2c_{3}e^{-A \langle \xi \rangle^{\frac{1}{s}}}|V_{2}(t, \xi)|\\
     &=c_{4}e^{-A \langle \xi \rangle^{\frac{1}{s}}}|V_{2}(t, \xi)|, \quad \text{where $c_{4}=2c_{3}.$ }
	\end{aligned}
 $$
 Without any loss of generality, we also  assume that $|V_{2}(t, \xi)|\geq1.$  Then  \eqref{derivative energy estimate case3 ch5} also can be express as 
\begin{equation}\label{modified derivative energy estimate case3 ch5}
	\begin{aligned}[b]
		\partial_{t} E_{\varepsilon}(t, \xi)
         &\leq \left(\partial_{t} P_{\varepsilon}(t) V(t, \xi), V(t, \xi)\right)+\varepsilon\langle\xi\rangle\left(P_{\varepsilon}(t) V(t,\xi), V(t,\xi)\right)\\
         &+c_{1}\langle\xi\rangle^{-1}\left(P_{\varepsilon}(t) V(t,\xi),V(t,\xi)\right)+c_{4}e^{-A \langle \xi \rangle^{\frac{1}{s}}}|V_{2}(t, \xi)|^2\\
         &\leq \left(\partial_{t} P_{\varepsilon}(t) V(t, \xi), V(t, \xi)\right)+\varepsilon\langle\xi\rangle\left(P_{\varepsilon}(t) V(t,\xi),V(t,\xi)\right)\\
        &+c_{1}\left(P_{\varepsilon}(t) V(t,\xi),V(t,\xi)\right)+c_{4}e^{-A \langle \xi \rangle^{\frac{1}{s}}}\left(P_{\varepsilon}(t) V(t,\xi), V(t,\xi)\right)\\
        &\leq \left(\frac{\left(\partial_{t} P_{\varepsilon}(t) V(t, \xi), V(t, \xi)\right)}{\left(P_{\varepsilon}(t) V(t, \xi), V(t, \xi)\right)}+\varepsilon\langle\xi\rangle+c_{1}+c_{4}\right)E_{\varepsilon}(t, \xi)\\
        &= \left(\frac{\left(\partial_{t} P_{\varepsilon}(t) V(t, \xi), V(t, \xi)\right)}{\left(P_{\varepsilon}(t) V(t, \xi), V(t, \xi)\right)}+\varepsilon\langle\xi\rangle+c_{5}\right)E_{\varepsilon}(t, \xi),
	\end{aligned}
 \end{equation}
 where $c_{5}=c_{1}+c_{4}.$ Again, from \cite{garetto micheal},  for all $\varepsilon \in(0,1]$ and $ t \in[0, T]$, we have the following estimate \begin{equation}\label{integral estimate case3 ch5}
		\int_{0}^{T} \frac{\left(\partial_{t} P_{\varepsilon}(t) V(t, \xi), V(t, \xi)\right)}{\left(P_{\varepsilon}(t) V(t, \xi), V(t, \xi)\right)} \mathrm{~d} t \leq C_{0} \varepsilon^{-\frac{2}{l}}\|a\|_{\mathcal{C}^{l}([0, T])}^{\frac{1}{l}}
	\end{equation}
 for some   constant $C_{0}>0.$  Now, applying   Gronwall's lemma, we get
  \begin{equation}\label{E epsilon estimate case3 ch5}
     E_\epsilon(t,\xi) \leq e^{(K_{0}\varepsilon^{-\frac{2}{l}}+\varepsilon\langle\xi\rangle+c_{5})T}E_\epsilon(0,\xi),
 \end{equation}
 where $K_{0}=C_{0}\|a\|_{\mathcal{C}^{l}([0, T])}^{\frac{1}{l}}.$ Let $\varepsilon^{-\frac{2}{l}}=\varepsilon\langle\xi\rangle$ and $\sigma=1+\frac{l}{2},$ then  $\varepsilon=\langle\xi\rangle^{\frac{-l}{2\sigma}}$ and $\varepsilon^{-\frac{2}{l}}=\langle\xi\rangle^{\frac{1}{\sigma}}.$ Thus  \eqref{E epsilon estimate case3 ch5} become
$$
\begin{aligned}
    E_\epsilon(t,\xi) &\leq e^{c_{5}T}e^{\frac{K_{00}}{2}(\langle\xi\rangle^{\frac{1}{\sigma}}+\langle\xi\rangle^{\frac{1}{\sigma}})}E_\epsilon(0,\xi)\leq c_{6} e^{K_{00}\langle\xi\rangle^{\frac{1}{\sigma}}}E_\epsilon(0,\xi),
\end{aligned}
$$
where $K_{00}= 2\max \{K_{0}, 1\}T$ and $c_{6}=e^{c_{5}T}.$ Consequently, from \eqref{E epsilon both side estimate case3 ch5}, we obtain
$$
c_{2}^{-1} \epsilon^{2}|V(t,\xi)|^{2} \leq  E_\epsilon(t,\xi)\leq c_{6} e^{K_{00}\langle\xi\rangle^{\frac{1}{\sigma}}}E_\epsilon(0,\xi)\leq c_{2}c_{6}\left|V(0,\xi)\right|^{2} e^{K_{00}\langle\xi\rangle^{\frac{1}{\sigma}}},
$$
which  also gives us
$$
|V(t,\xi)|^{2}\leq c_{2}^2 c_{6}^2 \langle\xi\rangle^{\frac{l}{\sigma}}e^{K_{00}\langle\xi\rangle^{\frac{1}{\sigma}}}\left|V(0,\xi)\right|^{2}
$$
for all $t\in[0,T]$ and   $\xi \in \mathcal{I}.$  This implies that 
$$\langle\xi\rangle^{2}|\widehat{v}(t, \xi)|^{2}+\left|\partial_{t} \widehat{v}(t, \xi)\right|^{2} \leq c_{2}^2 c_{6}^2 \langle\xi\rangle^{\frac{l}{\sigma}}e^{K_{00}\langle\xi\rangle^{\frac{1}{\sigma}}}\left(\langle\xi\rangle^{2}|\widehat{v}(0, \xi)|^{2}+\left|\partial_{t} \widehat{v}(0, \xi)\right|^{2}\right).$$
Let  us assume  $s < \sigma$. Since the initial Cauchy data $v_{0}, v_{1} \in \gamma_{\mathcal{H}}^{s}$,    we can conclude that the Cauchy problem \eqref{cauchy problem ch5} has a unique solution $v \in C^2\left([0, T] ; \gamma_{\mathcal{H}}^{s}\right)$ provided  
 $$1\leq s < 1+ \frac{l}{2}.$$
\end{proof}
Now we prove our fourth key result in the scenario when the propagation speed  $a \in \mathcal{C}^\alpha([0, T])$ with $0<\alpha<2.$ 
\begin{proof}[Proof of Theorem \ref{classical sol case4 ch5} \textbf{(Case 4)}]
In this case we assume that $a \in \mathcal{C}^\alpha([0, T])$ with $0<\alpha<2,$ be such that $a(t)\geq0.$ Here the roots   $\pm\sqrt{a(t)}$ may coincide and are H\"{o}lder continuous of order $\alpha/2$.  Using the same techniques demonstrated in the proof of Theorem \ref{classical sol case2 ch5} for $\sqrt{a(t)}$ (instead for $a(t)$), we get the desired result.
\end{proof}
\section{Very weak solutions}\label{very weak solutions section ch5}
In this section, we  investigate the existence of very weak solutions   to the
Cauchy problem \eqref{cauchy problem ch5} with distributional coefficients. We allow coefficients $a$ and $q$ to be irregular. As previously mentioned in Section \ref{sec1 ch5},  we possess a concept of very weak solutions for solving problems that may lack a significant solution in the usual distributional sense and discuss the corresponding results for $a, q \in \mathcal{D}^{\prime}([0, T])$. In this case, using the Friedrichs-mollifier, i.e., $\psi \in C_{0}^{\infty}(\mathbb{R}^n), \psi \geq 0$ and $\int_{\mathbb{R}^n} \psi=1$,   we regularize the distributional coefficient $a$ to derive the families of smooth functions $\left(a_{\varepsilon}\right)_{\varepsilon}$, namely
	$$
	a_{\varepsilon}(t):=\left(a * \psi_{\omega(\varepsilon)}\right)(t), \quad \psi_{\omega(\varepsilon)}(t)=\frac{1}{\omega(\varepsilon)} \psi\left(\frac{t}{\omega(\varepsilon)}\right), \quad \varepsilon \in(0,1],
	$$
	where $\omega(\varepsilon) \geq 0$ and $\omega(\varepsilon) \rightarrow 0$ as $\varepsilon \rightarrow 0$.
	\begin{defi}\label{moderate defi ch5}
		Let $K \Subset \mathbb{R}$ denote $K$ as a compact set in $\mathbb{R}$.  
		\begin{enumerate}[(i)]
			\item A net $\left(a_{\varepsilon}\right)_{\varepsilon} \in L_{m}^{\infty}(\mathbb{R})^{(0,1]}$ is said to be $L_{m}^{\infty}$-moderate if for all $K \Subset \mathbb{R}$, there exist $N \in \mathbb{N}_{0}$ and $C_{K}>0$ such that
			
			$$
			\left\|\partial^{k} a_{\varepsilon}\right\|_{L^{\infty}(K)} \leq C_K \varepsilon^{-N-k}, \quad \text {for  } k=0,1,\ldots, m,
			$$
			for all $\varepsilon \in(0,1]$.\\
			
			\item A net $\left(a_{\varepsilon}\right)_{\varepsilon} \in L_{m}^{\infty}(\mathbb{R})^{(0,1]}$ is said to be $L_{m}^{\infty}$-negligible if for all $K \Subset \mathbb{R}$ and $q \in \mathbb{N}_{0}$, there exists $C_{K}>0$ such that
			
			$$
			\left\|\partial^{k} a_{\varepsilon}\right\|_{L^{\infty}(K)} \leq C_{K} \varepsilon^{q}, \quad \text { for } k=0,1,\ldots, m,
			$$
			for all $\varepsilon \in(0,1]$.\\
			
			\item A net $\left(v_{\varepsilon}\right)_{\varepsilon} \in L^{2}\left([0, T] ; \mathrm{H}_{\mathcal{H}}^{s}\right)^{(0,1]}$ is said to be $L^{2}\left([0, T] ; \mathrm{H}_{\mathcal{H}}^{s}\right)$-moderate if there exist $N \in \mathbb{N}_{0}$ and $C>0$ such that
			$$
			\left\|v_{\varepsilon}\right\|_{L^{2}\left([0, T] ; H_{\mathcal{H}}^{s}\right)} \leq C \varepsilon^{-N}
			$$
			for all $\varepsilon \in(0,1]$.\\
			
			\item A net $\left(v_{\varepsilon}\right)_{\varepsilon} \in L^{2}\left([0, T] ; \mathrm{H}_{\mathcal{H}}^{s}\right)^{(0,1]}$ is said to be $L^{2}\left([0, T] ; \mathrm{H}_{\mathcal{H}}^{s}\right)$-negligible if for all $q \in \mathbb{N}_{0}$ there exists $C>0$ such that
			$$
			\left\|v_{\varepsilon}\right\|_{L^{2}\left([0, T] ; H_{\mathcal{H}}^{s}\right)} \leq C \varepsilon^{q}
			$$
			for all $\varepsilon \in(0,1]$.
		\end{enumerate}
		
	\end{defi}
	It is important to emphasize that the conditions of moderateness are natural in the sense that regularizations of distributions are moderate. Furthermore, by the structure theorems for distributions, we have
	$$
	\text { compactly supported distributions } \mathcal{E}^{\prime}(\mathbb{R}) \subset\left\{L^{2} \text {-moderate families}\right\}.
	$$
	Therefore, it may be possible that the Cauchy problem \eqref{cauchy problem ch5} does not have a solution in compactly supported distributions $\mathcal{E}^{\prime}(\mathbb{R})$.  To overcome the difficulty described above,   we introduce the concept of a very weak solution for the Cauchy problem \eqref{cauchy problem ch5} in the following manner:
	\begin{defi}\label{very weak sol defi ch5}
		Let $s \in \mathbb{R},$  and $\left(v_{0}, v_{1}\right) \in \mathrm{H}_{\mathcal{H}}^{1+s} \times \mathrm{H}_{\mathcal{H}}^{s}$. The net $\left(v_{\varepsilon}\right)_{\varepsilon} \in L^{2}\left([0, T] ; \mathrm{H}_{\mathcal{H}}^{1+s}\right)^{(0,1]}$ is a very weak solution of order $s$ to the Cauchy problem \eqref{cauchy problem ch5} if there exist
		\begin{enumerate}
			\item $L_{1}^{\infty}$-moderate regularisation $a_{\varepsilon}$ of the coefficient $a$,  
			
			\item $L^{\infty}$-moderate regularisation $q_{\varepsilon}$ of the coefficient $q$,
			
			\item $L^{2}\left([0, T] ; \mathrm{H}_{\mathcal{H}}^{s}\right)$-moderate regularisation $f_{\varepsilon}$ of the source term $f$,
		\end{enumerate}
		such that $\left(v_{\varepsilon}\right)_{\varepsilon}$ solves the regularised problem
		
		\begin{equation}\label{ regularized cauchy problem ch5}
			\left\{\begin{array}{l}
				\partial_{t}^{2} v_{\varepsilon}(t, x)+a_{\varepsilon}(t) \mathcal{H} v_{\varepsilon}(t, x)+q_{\varepsilon}(t) v_{\varepsilon}(t, x)=f_{\varepsilon}(t, x), \quad(t, x) \in(0, T] \times \mathbb{R}^n,\\
				v_{\varepsilon}(0, x)=v_{0}(x), \quad x \in \mathbb{R}^n, \\
				\partial_{t} v_{\varepsilon}(0, x)=v_{1}(x), \quad x \in \mathbb{R}^n,
			\end{array}\right.
		\end{equation}
		for all $\varepsilon \in(0,1]$ and is $L^{2}\left([0, T] ; \mathrm{H}_{\mathcal{H}}^{1+s}\right)$-moderate.
	\end{defi}
	We say that a distribution $a$ is positive if $\langle a, \psi\rangle \geq 0$ whenever $\psi \in$ $C_{0}^{\infty}(\mathbb{R})$ satisfying $\psi \geq 0$. Also, a distribution $a$ is strictly positive if there exists a constant $\alpha>0$ such that $a-\alpha$ is a positive distribution. Now, the existence theorem for the Cauchy problem \eqref{cauchy problem ch5} with distributional coefficients can now be formulated as follows:
	\begin{theorem}\textbf{(Existence)}\label{very weak sol ch5}
		Let $s \in \mathbb{R}$. Let the coefficients $a$ and $q$ be positive distributions with supports contained in $[0, T]$ such that $a \geq a_{0}>0$ for some constant $a_{0}$ and the initial Cauchy data $\left(v_{0}, v_{1}\right) \in \mathrm{H}_{\mathcal{H}}^{1+s} \times \mathrm{H}_{\mathcal{H}}^{s}$. Let $f(\cdot, x)$ be a distribution whose support is in $[0, T]$ for all $x \in \mathbb{R}^n$. Then the Cauchy problem \eqref{cauchy problem ch5} has a very weak solution of order $s$.
	\end{theorem}
	Moreover, one can understood the uniqueness of the very weak solution in the sense that if there is negligible modification in the approximations of the coefficients $a, q$, then it has negligible effect on the family of very weak solutions.
	\begin{defi}\label{unique sol defi ch5}
		We say that the Cauchy problem \eqref{cauchy problem ch5} has a $L^{2}\left([0, T] ; \mathrm{H}_{\mathcal{H}}^{1+s}\right)$-unique very weak solution, if
		
		\begin{enumerate}
			\item for all $L_{1}^{\infty}$-moderate nets $a_{\varepsilon}, \tilde{a}_{\varepsilon}$ such that $\left(a_{\varepsilon}-\tilde{a}_{\varepsilon}\right)_{\varepsilon}$ is $L_{1}^{\infty}$-negligible,
			
			\item for all $L^{\infty}$-moderate nets $q_{\varepsilon}, \tilde{q}_{\varepsilon}$ such that $\left(q_{\varepsilon}-\tilde{q}_{\varepsilon}\right)_{\varepsilon}$ is $L^{\infty}$-negligible; and
			
			\item or all $L^{2}\left([0, T] ; \mathrm{H}_{\mathcal{H}}^{s}\right)$-moderate nets $f_{\varepsilon}, \tilde{f}_{\varepsilon}$ such that $\left(f_{\varepsilon}-\tilde{f}_{\varepsilon}\right)_{\varepsilon}$ is $L^{2}\left([0, T] ; \mathrm{H}_{\mathcal{H}}^{s}\right)$-negligible,
		\end{enumerate}
		the net $\left(v_{\varepsilon}-\tilde{v}_{\varepsilon}\right)_{\varepsilon}$ is $L^{2}\left([0, T] ; \mathrm{H}_{\mathcal{H}}^{1+s}\right)$-negligible, where $\left(v_{\varepsilon}\right)_{\varepsilon}$ and $\left(\tilde{v}_{\varepsilon}\right)_{\varepsilon}$ are the families of solutions corresponding to the $\varepsilon$-parametrised problems
		\begin{equation}\label{parametrised solution ch5}
			\left\{\begin{array}{l}
				\partial_{t}^{2} v_{\varepsilon}(t, x)+a_{\varepsilon}(t) \mathcal{H} v_{\varepsilon}(t, x)+q_{\varepsilon}(t) v_{\varepsilon}(t, x)=f_{\varepsilon}(t, x), \quad(t, x) \in(0, T] \times \mathbb{R}^n,\\
				v_{\varepsilon}(0, x)=v_{0}(x), \quad x \in \mathbb{R}^n, \\
				\partial_{t} v_{\varepsilon}(0, x)=v_{1}(x), \quad x \in \mathbb{R}^n,
			\end{array}\right.
		\end{equation}
		and
		\begin{equation}\label{epsilon-parametrised solution ch5}
			\left\{\begin{array}{l}
				\partial_{t}^{2} \tilde{v}_{\varepsilon}(t, x)+\tilde{a}_{\varepsilon}(t) \mathcal{H} \tilde{v}_{\varepsilon}(t, x)+\tilde{q}_{\varepsilon}(t) \tilde{v}_{\varepsilon}(t, x)=\tilde{f}_{\varepsilon}(t, x), \quad(t, x) \in(0, T] \times \mathbb{R}^n, \\
				\tilde{v}_{\varepsilon}(0, x)=v_{0}(x), \quad x \in \mathbb{R}^n, \\
				\partial_{t} \tilde{v}_{\varepsilon}(0, x)=v_{1}(x), \quad x \in \mathbb{R}^n,
			\end{array}\right.
		\end{equation}
		respectively.	
	\end{defi}
	The next theorem establishes the uniqueness of the very weak solution to the Cauchy problem \eqref{cauchy problem ch5} in the sense of Definition \ref{unique sol defi ch5}.
	\begin{theorem}\textbf{(Uniqueness)}\label{uniqueness of sol ch5}
		Assume that $a,q$ be positive distributions with supports contained in $[0, T]$ such that $a \geq a_{0}>0$ for some constant $a_{0}$ and the initial Cauchy data $\left(v_{0}, v_{1}\right) \in \mathrm{H}_{\mathcal{H}}^{1+s} \times \mathrm{H}_{\mathcal{H}}^{s}$ for some $s \in \mathbb{R}$. Let $f(\cdot, x)$ be a distribution whose support is in $[0, T]$ for all $x \in \mathbb{R}^n$. Then the very weak solution of the Cauchy problem \eqref{cauchy problem ch5} is $L^{2}\left([0, T] ; \mathrm{H}_{\mathcal{H}}^{1+s}\right)$-unique.
	\end{theorem}
	We now provide the consistency result, which indicates that the weak solutions in Theorem \ref{very weak sol ch5} are able to capture the classical solutions provided by Theorem \ref{classical sol case1 ch5}, assuming the latter exist.
	\begin{theorem}\textbf{(Consistency)}\label{consistency of sol ch5}
		Let $s \in \mathbb{R}$ and $f \in L^{2}\left([0, T] ; \mathrm{H}_{\mathcal{H}}^{s}\right)$. Assume that $a \in L_{1}^{\infty}([0, T])$ be such that $a \geq a_{0}>0$, and $q \in L^{\infty}([0, T])$ be a positive distribution, and the initial Cauchy data $\left(v_{0}, v_{1}\right) \in \mathrm{H}_{\mathcal{H}}^{1+s} \times \mathrm{H}_{\mathcal{H}}^{s}$. Let $v$ be a very weak solution of \eqref{cauchy problem ch5}. Then for any regularising families $a_{\varepsilon}, q_{\varepsilon},f_{\varepsilon}$ in Definition \ref{moderate defi ch5}, the very weak solution $\left(v_{\varepsilon}\right)_{\varepsilon}$ converges to the classical solution of the Cauchy problem \eqref{cauchy problem ch5} in $L^{2}\left([0, T] ; \mathrm{H}_{\mathcal{H}}^{1+s}\right)$ as $\varepsilon \rightarrow 0$.
	\end{theorem}
We start by proving the existence of the very weak solution in the sense of Definition \ref{very weak sol defi ch5}.
\begin{proof}[Proof of Theorem \ref{very weak sol ch5}]
	For a fix $\varepsilon \in (0,1]$, consider the following regularised Cauchy problem
	\begin{equation}\label{regularised cauchy problem ch5}
		\left\{\begin{array}{l}
			\partial_{t}^{2} v_{\varepsilon}(t, x)+a_{\varepsilon}(t) \mathcal{H} v_{\varepsilon}(t, k)+q_{\varepsilon}(t) v_{\varepsilon}(t, x)=f_{\varepsilon}(t, x), \quad(t, x) \in(0, T] \times\mathbb{R}^n, \\
			v_{\varepsilon}(0, x)=v_{0}(x), \quad x \in  \mathbb{R}^n, \\
			\partial_{t} v_{\varepsilon}(0, k)=v_{1}(x), \quad x \in  \mathbb{R}^n.
		\end{array}\right.
	\end{equation}
Considering $\mathcal{H}$-Fourier transform with respect to  $x \in \mathbb{R}^n$ and   using the similar transformation as in  Theorem \ref{classical sol case1 ch5}, we see that 
\eqref{regularised cauchy problem ch5} is equivalent to the following first order system
\begin{equation}\label{equivalent regularised cauchy problem ch5}
	\left\{\begin{array}{l}
		\partial_{t} V_{\varepsilon}(t, \xi)=i\langle\xi\rangle A_{\varepsilon}(t) V_{\varepsilon}(t, \xi)+i\langle\xi\rangle^{-1} Q_{\varepsilon}(t) V_{\varepsilon}(t, \xi)+F_{\varepsilon}(t, \xi), \quad(t, \xi) \in(0, T] \times \mathcal{I}, \\
		V_{\varepsilon}(0, \xi)=V_{0}(\xi), \quad \xi \in \mathcal{I},
	\end{array}\right.
\end{equation}
where
\begin{equation}\label{regularised transformation ch5}
	V_{\varepsilon}(t, \xi):=\left(\begin{array}{c}
		i\langle\xi\rangle \widehat{v}_{\varepsilon}(t, \xi) \\
		\partial_{t} \widehat{v}_{\varepsilon}(t, \xi)
	\end{array}\right), \quad V_{0}(\xi):=\left(\begin{array}{c}
		i\langle\xi\rangle \widehat{v}_{0}(\xi) \\
		\widehat{v}_{1}(\xi)
	\end{array}\right),
\end{equation}
and
\begin{equation}\label{A_epsilon Q_epsilon defi ch5}
	A_{\varepsilon}(t):=\left(\begin{array}{cc}
		0 & 1 \\
		a_{\varepsilon}(t) & 0
	\end{array}\right), ~~ Q_{\varepsilon}(t):=\left(\begin{array}{cc}
		0 & 0 \\
		q_{\varepsilon}(t)-a_{\varepsilon}(t) & 0
	\end{array}\right),   F_{\varepsilon}(t, \xi):=\left(\begin{array}{c}
0 \\
\widehat{f}_{\varepsilon}(t, \xi)
\end{array}\right).
\end{equation}
Note that the matrix $A_{\varepsilon}(t)$ has eigenvalues $\pm \sqrt{a_{\varepsilon}(t)}$ and its symmetriser $S_{\varepsilon}$ is given by
\begin{equation}
	S_{\varepsilon}(t)=\left(\begin{array}{cc}
		a(t) & 0 \\
		0 & 1
	\end{array}\right),
	\end{equation}
i.e.,  
$$
S_{\varepsilon} A_{\varepsilon}-A_{\varepsilon}^{*} S_{\varepsilon}=0 .
$$
Also, for each $t\in [0,T]$, the matrix  $S_{\varepsilon}(t)$ is positive  and 
\begin{equation}\label{regularised symmetriser norm ch5}
	\|S_{\varepsilon}(t)\| \leq(1+|a_{\varepsilon}(t)|).
\end{equation}  
Again, from the definition of $S_{\varepsilon}$ and $Q_{\varepsilon}$, we have
$$
\partial_{t} S_{\varepsilon}(t):=\left(\begin{array}{cc}
	\partial_{t} a_{\varepsilon}(t) & 0 \\
	0 & 0
\end{array}\right) \text { and }\left(S_{\varepsilon} Q_{\varepsilon}-Q_{\varepsilon}^{*} S\right)(t):=\left(\begin{array}{cc}
	0 & a_{\varepsilon}(t)-q_{\varepsilon}(t) \\
	q_{\varepsilon}(t)-a_{\varepsilon}(t) & 0
\end{array}\right).
$$
Therefore 
\begin{equation}\label{regularised SQ-QS estimate ch5}
	\left\|\partial_{t}S_{\varepsilon}(t)\right\| \leq\left|\partial_{t} a_{\varepsilon}(t)\right|~~ \text{and}~~ \left\|\left(S_{\varepsilon} Q_{\varepsilon}-Q_{\varepsilon}^{*} S_{\varepsilon}\right)(t)\right\| \leq|q_{\varepsilon}(t)|+|a_{\varepsilon}(t)|, \quad \text{for all}~~ t\in[0, T].
\end{equation}
Considering  the energy equation
$$
E_{\varepsilon}(t, \xi):=(S_{\varepsilon}(t) V_{\varepsilon}(t, \xi), V_{\varepsilon}(t, \xi))
$$
and following the techniques used in \eqref{upper energy estimate ch5} and \eqref{lower energy estimate ch5},  we   obtain
\begin{equation}\label{regularised energy estimate ch5}
	\inf _{t \in[0, T]}\left\{a_{\varepsilon}(t), 1\right\}\left|V_{\varepsilon}(t, \xi)\right|^{2} \leq E_{\varepsilon}(t, \xi) \leq \sup _{t \in[0, T]}\left\{a_{\varepsilon}(t), 1\right\}\left|V_{\varepsilon}(t, \xi)\right|^{2}.
\end{equation}
Since $\psi \in$ $C_{0}^{\infty}(\mathbb{R}), \psi \geq 0, \operatorname{supp}(\psi) \subseteq K$ and $a,q$ are positive distributions with compact supports contained in $[0, T]$,  throughout the section, we can assume that $K=[0, T]$. By the structure theorem for compactly supported distributions, there exist $L_{1}, L_{2} \in \mathbb{N}$ and $c_{1}, c_{2}>0$ such that
\begin{equation}\label{regularised a,q estimate ch5}
	\left|\partial_{t}^{k} a_{\varepsilon}(t)\right| \leq c_{1} \omega(\varepsilon)^{-L_{1}-k} \quad \text { and } \quad\left|\partial_{t}^{k} q_{\varepsilon}(t)\right| \leq c_{2} \omega(\varepsilon)^{-L_{2}-k}
\end{equation}
for all $k \in \mathbb{N}_{0}$ and $t \in[0, T]$. Moreover, we can write
\begin{equation}\label{regularised a_{epsilon} inequality ch5}
\begin{aligned}[b]
	a_{\varepsilon}(t) & =\left(a * \psi_{\omega(\varepsilon)}\right)(t)=\int_{\mathbb{R}} a(t-\tau) \psi_{\omega(\varepsilon)}(\tau) \mathrm{d} \tau\\&=\int_{\mathbb{R}} a(t-\omega(\varepsilon) \tau) \psi(\tau) \mathrm{d} \tau \\
	& =\int_{\mathrm{K}} a(t-\omega(\varepsilon) \tau) \psi(\tau) \mathrm{d} \tau\\& \geq a_0 \int_{\mathrm{K}} \psi(\tau) \mathrm{d} \tau:=\tilde{a}_0>0, \quad \text{for all}~~ t \in[0, T],
\end{aligned}
\end{equation}
where the positive  constant $\tilde{a}_0$ is  independent of $\varepsilon$. Now, combining    \eqref{regularised a,q estimate ch5} and \eqref{regularised a_{epsilon} inequality ch5}, we get
\begin{equation}\label{regularised energy weighted estimate ch5}
	c_{00}\left|V_{\varepsilon}(t, \xi)\right|^{2} \leq E_{\varepsilon}(t, \xi) \leq\left(1+c_{1} \omega(\varepsilon)^{-L_{1}}\right)\left|V_{\varepsilon}(t, \xi)\right|^{2},
\end{equation}
where $c_{00}$ is a positive constant. Further, using  the estimates \eqref{regularised symmetriser norm ch5}, \eqref{regularised SQ-QS estimate ch5}, and \eqref{regularised a,q estimate ch5}, we obtain
\begin{equation}\label{regularised derivative energy estimate ch5}
	\partial_{t} E_{\varepsilon}(t, \xi)\leq c_{0}\left(1+\omega(\varepsilon)^{-L_{1}-1}+\omega(\varepsilon)^{-L_{1}}+ \omega(\varepsilon)^{-L_{2}}\right) E_{\varepsilon}(t, \xi)+c^{\prime}\left(1+\omega(\varepsilon)^{-L_{1}}\right)|F_{\varepsilon}(t, \xi)|^{2}
\end{equation}
for some constants $c_{0},c^{\prime}>0.$   Applying the Gronwall’s lemma to   \eqref{regularised derivative energy estimate ch5}, we can conclude that
\begin{equation}\label{regularised energy estimate relation ch5}
	E_{\varepsilon}(t, \xi) \leq e^{c_{0}T \left(1+\omega(\varepsilon)^{-L_{1}}+ \omega(\varepsilon)^{-L_{2}}\right)} \left(E_{\varepsilon}(0, \xi)+c^{\prime}\left(1+\omega(\varepsilon)^{-L_{1}}\right) \int_{0}^{t}\left|F_{\varepsilon}(\tau, \xi)\right|^{2} \mathrm{~d} \tau\right).
\end{equation}
Thus, combining  \eqref{regularised energy weighted estimate ch5} and \eqref{regularised energy estimate relation ch5}, we obtain
$$
\begin{aligned}
	& c_{00}\left|V_{\varepsilon}(t, \xi)\right|^{2} \leq E_{\varepsilon}(t, \xi) \\
	& \leq e^{c_{0}T \left(1+\omega(\varepsilon)^{-L_{1}-1}+\omega(\varepsilon)^{-L_{1}}+ \omega(\varepsilon)^{-L_{2}}\right)} \left( \left(1+c_{1} \omega(\varepsilon)^{-L_{1}}\right)\left|V_{\varepsilon}(0, \xi)\right|^{2}+c^{\prime}\left(1+\omega(\varepsilon)^{-L_{1}}\right) \int_{0}^{T}\left|F_{\varepsilon}(\tau, \xi)\right|^{2} \mathrm{~d} \tau\right) \\
	& \leq c^{\prime\prime} e^{c_{0}T \left(1+\omega(\varepsilon)^{-L_{1}-1}+\omega(\varepsilon)^{-L_{1}}+ \omega(\varepsilon)^{-L_{2}}\right)} \times e^{ \omega(\varepsilon)^{-L_{1}}}\left(\left|V_{\varepsilon}(0, \xi)\right|^{2}+ \int_{0}^{T}\left|F_{\varepsilon}(\tau, \xi)\right|^{2} \mathrm{~d} \tau\right) \\
	& \leq c^{\prime\prime} e^{c_{0}T} e^{C \omega(\varepsilon)^{-L}}\left(\left|V_{\varepsilon}(0, \xi)\right|^{2}+ \int_{0}^{T}\left|F_{\varepsilon}(\tau, \xi)\right|^{2} \mathrm{~d} \tau\right),
\end{aligned}
$$
where $c^{\prime\prime}= \max\{1,c_{1}, c^{\prime}\}$, $C= 4\max\{c_{0}T,1\}$ and $L=\max\{L_1+1,L_2\}$. This gives us
$$
\left|V_{\varepsilon}(t, \xi)\right|^{2} \leq c^{\prime\prime} e^{c_{0}T} c_{00}^{-1} e^{C \omega(\varepsilon)^{-L}}\left(\left|V_{\varepsilon}(0, \xi)\right|^{2}+ \int_{0}^{T}\left|F_{\varepsilon}(\tau, \xi)\right|^{2} \mathrm{~d} \tau\right).
$$
Letting $\omega(\varepsilon)\sim (\log (\varepsilon^{\frac{-L}{C}}))^{\frac{-1}{L}}$ and using the transformation considered in \eqref{regularised transformation ch5}, we get
\begin{equation}\label{regularised epsilon estimate ch5}
		\langle\xi\rangle^{2}\left|\widehat{v}_{\varepsilon}(t, \xi)\right|^{2}+ \left|\partial_{t} \widehat{v}_{\varepsilon}(t, \xi)\right|^{2} \lesssim \varepsilon^{-L}\left(\langle\xi\rangle^{2}\left|\widehat{v}_{0}(\xi)\right|^{2}+\left|\widehat{v}_{1}(\xi)\right|^{2}+\int_{0}^{T}\left|\widehat{f}_{\varepsilon}(\tau, \xi)\right|^{2} \mathrm{~d} \tau\right).
\end{equation}
Now multiply \eqref{regularised epsilon estimate ch5} by $\langle\xi\rangle^{2s}$, and then by applying Plancherel’s formula, we conclude that
\begin{equation}\label{regularised weak solution estimate ch5}
	\|v_{\varepsilon}(t, \cdot)\|_{\mathrm{H}_{\mathcal{H}}^{1+s}}^{2}+\left\|\partial_t{v_{\varepsilon}}(t, \cdot)\right\|_{\mathrm{H}_{\mathcal{H}}^{s}}^{2} \lesssim \varepsilon^{-L} \left(\left\|v_{0}\right\|_{\mathrm{H}_{\mathcal{H}}^{1+s}}^{2}+\left\|v_{1}\right\|_{\mathrm{H}_{\mathcal{H}}^{s}}^{2}+\|f_{\varepsilon}\|_{L^{2}\left([0, T] ; \mathrm{H}_{\mathcal{H}}^{s}\right)}^{2}\right),
\end{equation}
for all $t \in[0, T]$. Since $\left(f_{\varepsilon}\right)_{\varepsilon}$ is a $L^{2}\left([0, T] ; \mathrm{H}_{\mathcal{H}}^{s}\right)$ moderate regularisation of $f$, thus we can find a constant $L_{0}\in \mathbb{N}_{0} $ such that
\begin{equation}\label{regularised source term estimate ch5}
	\|f_{\varepsilon}\|_{L^{2}\left([0, T] ; \mathrm{H}_{\mathcal{H}}^{s}\right)}^{2} \lesssim \varepsilon^{-L_{0}}.
\end{equation}
Now, integrating the estimate \eqref{regularised weak solution estimate ch5} with respect to      $t \in[0, T]$ variable and combining it with \eqref{regularised source term estimate ch5}, we deduce that
$$
\left\|v_{\varepsilon}\right\|_{L^{2}\left([0, T] ; \mathrm{H}_{\mathcal{H}}^{1+s}\right)} \lesssim \varepsilon^{-L-L_{0}} \text { and }\left\|\partial_{t} v_{\varepsilon}\right\|_{L^{2}\left([0, T] ; \mathrm{H}_{\mathcal{H}}^{s}\right)} \lesssim \varepsilon^{-L-L_{0}-1}.
$$
This implies that  $v_{\varepsilon}$ is $L^{2}\left([0, T] ; \mathrm{H}_{\mathcal{H}}^{1+s}\right)$-moderate and  this completes the proof of the theorem.
\end{proof}
In the next result, we  show the uniqueness of very weak solutions in the context of Definition \ref{unique sol defi ch5}.
\begin{proof}[Proof of Theorem \ref{uniqueness of sol ch5}]
	 Let $\left(v_{\varepsilon}\right)_{\varepsilon}$and $\left(\tilde{v}_{\varepsilon}\right)_{\varepsilon}$ solves the Cauchy problems \eqref{parametrised solution ch5} and \eqref{epsilon-parametrised solution ch5}, respectively. Then 
	\begin{eqnarray}\label{uniqueness solution difference ch5}
		\left\{\begin{array}{l}
			\partial_{t}^{2} (v_{\varepsilon}-\tilde{v}_{\varepsilon})(t, x)+a_{\varepsilon}(t) \mathcal{H} (v_{\varepsilon}-\tilde{v}_{\varepsilon})(t, x)+q_{\varepsilon}(t) (v_{\varepsilon}-\tilde{v}_{\varepsilon})(t, x)=h_{\varepsilon}(t, x), \quad(t, x) \in(0, T] \times \mathbb{R}^n, \\
			(v_{\varepsilon}-\tilde{v}_{\varepsilon})(0, x)=0, \quad x \in \mathbb{R}^n,\\
			(\partial_{t} v_{\varepsilon}-\partial_{t} \tilde{v}_{\varepsilon})(0, x)=0, \quad x \in \mathbb{R}^n,
		\end{array}\right.
	\end{eqnarray}	
	where
	$$
	h_{\varepsilon}(t, x):=\left(\tilde{a}_{\varepsilon}-a_{\varepsilon}\right)(t) \mathcal{H}\tilde{v}_{\varepsilon}(t, x)+\left(\tilde{q}_{\varepsilon}-q_{\varepsilon}\right)(t) \tilde{v}_{\varepsilon}(t, x)+\left(f_{\varepsilon}-\tilde{f}_{\varepsilon}\right)(t, x) .
	$$
	Since $\left(a_{\varepsilon}-\tilde{a}_{\varepsilon}\right)_{\varepsilon}$ is $L_{1}^{\infty}$-negligible, $\left(q_{\varepsilon}-\tilde{q}_{\varepsilon}\right)_{\varepsilon}$ is $L^{\infty}$-negligible, and $\left(f_{\varepsilon}-\tilde{f}_{\varepsilon}\right)_{\varepsilon}$ is $L^{2}\left([0, T] ; \mathrm{H}_{\mathcal{H}}^{s}\right)$-negligible, it follows that $h_{\varepsilon}$ is $L^{2}\left([0, T] ; \mathrm{H}_{\mathcal{H}}^{s}\right)$-negligible. Let us denote $w_{\varepsilon}(t, x):=v_{\varepsilon}(t, x)-\tilde{v}_{\varepsilon}(t, x)$. Taking  $\mathcal{H}$-Fourier transform with respect to  $x \in \mathbb{R}^n$ and     using the similar transformation as in  Theorem \ref{very weak sol ch5}, we see that the first order system corresponding to \eqref{uniqueness solution difference ch5} becomes
	\begin{equation}
			\left\{\begin{array}{l}
			\partial_{t} W_{\varepsilon}(t, \xi)=i\langle\xi\rangle A_{\varepsilon}(t) W_{\varepsilon}(t, \xi)+i\langle\xi\rangle^{-1} Q_{\varepsilon}(t) W_{\varepsilon}(t, \xi)+H_{\varepsilon}(t, \xi), \quad \xi \in \mathcal{I},\\
			W_{\varepsilon}(0, \xi)=0, \quad \xi \in \mathcal{I},
		\end{array}\right.
	\end{equation}	
	where $A_{\varepsilon}(t), Q_{\varepsilon}(t)$ are given in \eqref{A_epsilon Q_epsilon defi ch5} and
	\begin{equation}\label{uniqueness transformation ch5}
		W_{\varepsilon}(t, \xi):=\left(\begin{array}{c}
			i\langle\xi\rangle \widehat{w}_{\varepsilon}(t, \xi) \\
			\partial_{t} \widehat{w}_{\varepsilon}(t, \xi)
		\end{array}\right),  \quad W_{0}(\xi):=\left(\begin{array}{c}
		0\\
		0
	\end{array}\right), \quad H_{\varepsilon}(t, \xi):=\left(\begin{array}{c}
	0 \\
	\widehat{h_{\varepsilon}}(t, \xi)
\end{array}\right).
	\end{equation}
For the symmetriser
$$
S_{\varepsilon}(t)=\left(\begin{array}{cc}
	a(t) & 0 \\
	0 & 1
\end{array}\right),
$$
define the energy 
$$
E_{\varepsilon}(t, \xi):=\left(S_{\varepsilon}(t) W_{\varepsilon}(t, \xi), W_{\varepsilon}(t, \xi)\right).
$$
Then,  similar to estimate \eqref{regularised energy weighted estimate ch5}, we  also  can find $c_{0},c_{1}>0$ and $L_{1} \in \mathbb{N}$ such that
\begin{equation}\label{uniqueness energy both side estimate ch5}
	c_{0}\left|W_{\varepsilon}(t, \xi)\right|^{2} \leq E_{\varepsilon}(t, \xi) \leq\left(1+c_{1} \omega(\varepsilon)^{-L_{1}}\right)\left|W_{\varepsilon}(t, \xi)\right|^{2}.
\end{equation}
Adapting  similar techniques used in  the proof of  Theorem \ref{classical sol case1 ch5}, we can observe that
	\begin{equation}\label{uniqueness derivative energy estimate ch5}
			\partial_{t} E_{\varepsilon}(t, \xi)
			\leq c_{00}\omega(\varepsilon)^{-L} (|W_{\varepsilon}(t, \xi)|^{2}+ |F_{\varepsilon}(t, \xi)|^2),
	\end{equation}
where $c_{00}$ is a positive constant and $L\in\mathbb{N}$ is same constnat as in Theorem \ref{regularised cauchy problem ch5}. Now,   applying the Gronwall’s lemma to  the inequality \eqref{uniqueness derivative energy estimate ch5} and  using \eqref{uniqueness energy both side estimate ch5}, we can conclude  
\begin{equation}\label{uniqueness gronwall lemma estimate ch5}
	\left|W_{\varepsilon}(t, \xi)\right|^{2} \leq c_{0}^{-1} e^{c_{00}T\omega(\varepsilon)^{-L}}\left(\left|W_{\varepsilon}(0, \xi)\right|^{2}+\int_{0}^{T}\left|H_{\varepsilon}(\tau, \xi)\right|^{2} \mathrm{~d} \tau\right).
\end{equation}
Let $\omega(\varepsilon)\sim (\log (\varepsilon^{\frac{-L}{c_{oo}}}))^{\frac{-1}{L}}$.  Knowing the fact  that $W_{\varepsilon}(0, \xi) \equiv 0$ for all $\varepsilon \in(0,1]$, we have 
$$
\left|W_{\varepsilon}(t, \xi)\right|^{2} \lesssim \varepsilon^{-L} \int_{0}^{T}\left|H_{\varepsilon}(\tau, \xi)\right|^{2} \mathrm{~d} \tau.
$$
Now using the transformation as in \eqref{uniqueness transformation ch5}, we get
\begin{equation}\label{uniqueness transformation epsilon estimate ch5}
	\langle\xi\rangle^{2}\left|\widehat{w}_{\varepsilon}(t, \xi)\right|^{2}+\left|\partial_{t} \widehat{w}_{\varepsilon}(t, \xi)\right|^{2} \lesssim  \varepsilon^{-L} \int_{0}^{T}\left|\widehat{h}_{\varepsilon}(\tau, \xi)\right|^{2} \mathrm{~d} \tau.
\end{equation}
Multiply \eqref{uniqueness transformation epsilon estimate ch5} by $\langle\xi\rangle^{2s}$ and then   applying Plancherel’s formula, one can conclude that
\begin{equation}\label{uniqueness plancherel estimate ch5}
	\|w_{\varepsilon}(t, \cdot)\|_{\mathrm{H}_{\mathcal{H}}^{1+s}}^{2}+\left\|\partial_t w_{\varepsilon}(t, \cdot)\right\|_{\mathrm{H}_{\mathcal{H}}^{s}}^{2} \lesssim \varepsilon^{-L} \left\|h_{\varepsilon}\right\|_{L^{2}\left([0, T] ; \mathrm{H}_{\mathcal{H}}^{s}\right)}^{2}
\end{equation}
for all $t \in[0, T]$. Using the fact that  $h_{\varepsilon}$ is $L^{2}\left([0, T] ; \mathrm{H}_{\mathcal{H}}^{s}\right)$-negligible,  we get
$$
\left\|w_{\varepsilon}(t, \cdot)\right\|_{\mathrm{H}_{\mathcal{H}}^{1+s}}^{2}+\left\|\partial_{t} w_{\varepsilon}(t, \cdot)\right\|_{\mathrm{H}_{\mathcal{H}}^{s}}^{2} \lesssim \varepsilon^{- L} \varepsilon^{L+q}=\varepsilon^{q}, 
$$
for all $ q \in \mathbb{N}_{0}$ and  $t \in[0, T]$. Taking integration  with respect to  the  $t \in[0, T]$ variable, we obtain
$$
\left\|w_{\varepsilon}\right\|_{L^{2}\left([0, T] ; H_{\mathcal{H}}^{1+s}\right)}^{2}+\left\|\partial_{t} w_{\varepsilon}\right\|_{L^{2}\left([0, T] ; H_{\mathcal{H}}^{s}\right)}^{2} \lesssim \varepsilon^{q}, \quad \text { for all } q \in \mathbb{N}_{0},
$$
with the constant independent of $t$. This implies that  $w_{\varepsilon}(t, x)=v_{\varepsilon}(t, x)-\tilde{v}_{\varepsilon}(t, x)$ is $L^{2}\left([0, T] ; \mathrm{H}_{\mathcal{H}_{\hbar, V}}^{1+s}\right)$-negligible and this completes the proof of the theorem.
\end{proof}
In the following result  we   show that very weak solutions are consistent.
\begin{proof}[Proof of Theorem \ref{consistency of sol ch5}]
	 Let $\tilde{v}$ represents the classical solution in Theorem \ref{classical sol case1 ch5}. By definition of the classic solution and very weak solution, we know that
	\begin{equation}\label{consistency classic sol ch5}
		\left\{\begin{array}{l}
			\partial_{t}^{2} \tilde{v}(t, x)+a(t) \mathcal{H} \tilde{v}(t,x)+q(t)\tilde{v}(t, x)=f(t, x),(t, x), \quad (t, x) \in (0, T] \times \mathbb{R}^n, \\
			\tilde{v}(0, k)=v_{0}(x), \quad x \in \mathbb{R}^n, \\
			\partial_{t} \tilde{v}(0, x)=v_{1}(x), \quad x \in \mathbb{R}^n,
		\end{array}\right.
	\end{equation}
and there exists a net $\left(v_{\varepsilon}\right)_{\varepsilon}$ such that
\begin{equation}\label{consistency very weak sol ch5}
\left\{\begin{array}{l}
	\partial_{t}^{2} v_{\varepsilon}(t, x)+a_{\varepsilon}(t) \mathcal{H} v_{\varepsilon}(t, x)+q_{\varepsilon}(t) v_{\varepsilon}(t, x)=f_{\varepsilon}(t, x),(t, x), \quad(t, x) \in(0, T] \times \mathbb{R}^n, \\
	v_{\varepsilon}(0, k)=u_{0}(x), \quad x \in \mathbb{R}^n, \\
	\partial_{t} v_{\varepsilon}(0, x)=u_{1}(x), \quad x \in \mathbb{R}^n.
\end{array}\right.
\end{equation}
 We also can rewrite \eqref{consistency classic sol ch5} as follows:
\begin{equation}\label{consistency modified classic sol ch5}
	\left\{\begin{array}{l}
		\partial_{t}^{2} \tilde{v}(t, x)+a_{\varepsilon}(t) \mathcal{H}\tilde{v}(t, x)+q_{\varepsilon}(t) \tilde{v}(t, x)=f_{\varepsilon}(t, x)+h_{\varepsilon}(t, x),(t, x) \in(0, T] \times \mathbb{R}^n, \\
		\tilde{v}(0, x)=v_{0}(x), \quad x \in \mathbb{R}^n,\\
		\partial_{t} \tilde{v}(0, x)=v_{1}(x), \quad x \in \mathbb{R}^n,
	\end{array}\right.
\end{equation}
where
$$
h_{\varepsilon}(t, x):=\left(a_{\varepsilon}-a\right)(t) \mathcal{H} \tilde{v}(t, x)+\left(q_{\varepsilon}-q\right)(t) \tilde{v}(t, x)+\left(f-f_{\varepsilon}\right)(t, x).
$$
For $a \in L_{1}^{\infty}([0, T])$, $b \in L^{\infty}([0, T])$, and $f \in L^{2}\left([0, T] ; \mathrm{H}_{\mathcal{H}}^{s}\right)$,    
 since the nets $\left(a_{\varepsilon}-a\right)_{\varepsilon}$, $\left(q_{\varepsilon}-q\right)_{\varepsilon}$, and $\left(f_{\varepsilon}-f\right)_{\varepsilon}$ are converging to 0 as $\varepsilon \rightarrow 0$,    then  $h_{\varepsilon} \in L^{2}\left([0, T] ; \mathrm{H}_{\mathcal{H}}^{s}\right)$ and $h_{\varepsilon} \rightarrow 0$ in $L^{2}\left([0, T] ; \mathrm{H}_{\mathcal{H}}^{s}\right)$ as $\varepsilon \rightarrow 0$. From \eqref{consistency very weak sol ch5} and \eqref{consistency modified classic sol ch5}, we get that $w_{\varepsilon}(t, x):=\left(\tilde{v}-v_{\varepsilon}\right)(t, x)$ solves the Cauchy problem
\begin{equation}\label{consistency modified very weak sol ch5}
	\left\{\begin{array}{l}
		\partial_{t}^{2} w_{\varepsilon}(t, x)+a_{\varepsilon}(t) \mathcal{H} w_{\varepsilon}(t, x)+q_{\varepsilon}(t) w_{\varepsilon}(t, x)=h_{\varepsilon}(t, x), \quad(t, x) \in(0, T] \times\mathbb{R}^n, \\
		w_{\varepsilon}(0, x)=0, \quad x \in \mathbb{R}^n, \\
		\partial_{t} w_{\varepsilon}(0, x)=0, \quad x \in \mathbb{R}^n.
	\end{array}\right.
\end{equation}
Now, one can obtain the following energy estimate using the similar techniques as in  the proof of the Theorem \ref{classical sol case1 ch5}
\begin{equation}\label{consistency energy estimate ch5}
	\begin{aligned}[b]
		\partial_{t} E_{\varepsilon}(t, \xi) \leq& \left(\left|\partial_{t} a_{\varepsilon}(t)\right|+|q_{\varepsilon}(t)|+|a_{\varepsilon}(t)|+1+\left| a_{\varepsilon}(t)\right|\right)|W_{\varepsilon}(t, \xi)|^{2}+(1+\left| a_{\varepsilon}(t)\right|)|H_{\varepsilon}(t, \xi)|^2\\
		\leq& c(E_{\varepsilon}(t, \xi) + |H_{\varepsilon}(t, \xi)|^2),
	\end{aligned}
\end{equation}
where $c$ is a positive constant. Applying the Gronwall’s lemma to  the inequality \eqref{consistency energy estimate ch5} and using the
energy bounds similar to \eqref{uniqueness energy both side estimate ch5}, we can conclude that
$$
\left|W_{\varepsilon}(t, \xi)\right|^{2} \lesssim\left|W_{\varepsilon}(0, \xi)\right|^{2}+\int_{0}^{T}\left|H_{\varepsilon}(\tau, \xi)\right|^{2} \mathrm{~d} \tau.
$$
Now,   applying Plancherel’s formula and using the fact that $W_{\varepsilon}(0, \xi) \equiv 0$ for all $\varepsilon \in(0,1]$, we have
$$
\left\|w_{\varepsilon}(t, \cdot)\right\|_{\mathrm{H}_{\mathcal{H}}^{1+s}}^{2}+\left\|\partial_{t} w_{\varepsilon}(t, \cdot)\right\|_{\mathrm{H}_{\mathcal{H}}^{s}}^{2} \lesssim\left\|h_{\varepsilon}\right\|_{L^{2}\left([0, T] ; \mathrm{H}_{\mathcal{H}}^{s}\right)}^{2}
$$
for all $t \in[0, T]$. Taking integration to the above estimate with respect to  $t \in[0, T]$  variable, we can write   
$$
\left\|w_{\varepsilon}\right\|_{L^{2}\left([0, T] ; H_{\mathcal{H}}^{1+s}\right)}^{2}+\left\|\partial_{t} w_{\varepsilon}\right\|_{L^{2}\left([0, T] ; H_{\mathcal{H}}^{s}\right)}^{2} \lesssim \left\|h_{\varepsilon}\right\|_{L^{2}\left([0, T] ; \mathrm{H}_{\mathcal{H}}^{s}\right)}^{2}
$$
with the constant independent of $t $. Since $h_{\varepsilon} \rightarrow 0$ in $L^{2}\left([0, T] ; \mathrm{H}_{\mathcal{H}}^{s}\right)$, then  
$$
w_{\varepsilon} \rightarrow 0 \text { in } L^{2}\left([0, T] ; \mathrm{H}_{\mathcal{H}}^{1+s}\right) \quad \text{as}~~ \varepsilon \rightarrow 0,
$$
i.e.,
$$
v_{\varepsilon} \rightarrow \tilde{v} \text { in } L^{2}\left([0, T] ; \mathrm{H}_{\mathcal{H}}^{1+s}\right) \quad  \text{as}~~ \varepsilon \rightarrow 0.
$$
Since, any two representations of $v$ will differ from $\left(v_{\varepsilon}\right)_{\varepsilon}$ by a $L^{2}\left([0, T] ; \mathrm{H}_{\mathcal{H}}^{1+s}\right)$-negligible net, hence the limit is same for every representation of $v$. This completes the proof of the theorem.
\end{proof}

\section*{Acknowledgement}
 The first author was supported by  Core Research Grant, RP03890G,  Science and Engineering
	Research Board (SERB), DST,  India. The second author is    supported by the
	institute assistantship from the Indian Institute of Technology Delhi, India. The last author is supported by the DST-INSPIRE Faculty Fellowship DST/INSPIRE/04/2023/002038.  
	

\end{document}